\documentclass[11pt]{amsart}
\usepackage{amsmath,amssymb,amsthm,bbm,graphicx}
\renewcommand{\epsilon}{\varepsilon}
\renewcommand{\phi}{\varphi}
\renewcommand{\kappa}{\varkappa}
\renewcommand{\rho}{\varrho}
\renewcommand{\tilde}{\widetilde}
\newcommand{\qu}{\mathrm{que}}
\newcommand{\an}{\mathrm{ann}}
\newcommand{\ct}{\mathrm{cr}}

\newcommand{\IE}{\mathbb{E}}
\newcommand{\I}{\mathbbm{1}}
\newcommand{\IN}{\mathbb{N}}
\newcommand{\IP}{\mathbb{P}}
\newcommand{\IR}{\mathbb{R}}
\newcommand{\IZ}{\mathbb{Z}}
\newcommand{\cal}{\mathcal}
\newtheorem{theorem}{Theorem}[section]
\newtheorem{lemma}[theorem]{Lemma}
\newtheorem{corollary}[theorem]{Corollary}
\newtheorem{proposition}[theorem]{Proposition}
\newtheorem{remark}[theorem]{Remark}
\numberwithin{equation}{section}

\begin{document}

\title[Crossing speeds of random walks among
``sparse'' or ``spiky'' potentials]
{Crossing speeds of random walks among ``sparse'' or ``spiky''
  Bernoulli potentials on integers}

\author{Elena Kosygina}
\address{Department of Mathematics, Baruch College, One Bernard Baruch
  Way, New York, NY 10010}
\email{elena.kosygina@baruch.cuny.edu}
\thanks{\textit{2000 Mathematics Subject Classification.}
  Primary: 60K37, 60F05, 60J80. Secondary: 60J60.}
\thanks{\textit{Key words:} random walk, random potential, speed,
  quenched, annealed.  }
\begin{abstract}
  We consider a random walk among i.i.d.\ obstacles on $\IZ$ under the
  condition that the walk starts from the origin and reaches a remote
  location $y$. The obstacles are represented by a killing potential,
  which takes value $M>0$ with probability $p$ and value $0$ with
  probability $1-p$, $0<p\le 1$, independently at each site of $\IZ$. We
  consider the walk under both quenched and annealed measures. It is
  known that under either measure the crossing time from $0$ to $y$ of
  such walk, $\tau_y$, grows linearly in $y$. More precisely, the
  expectation of $\tau_y/y$ converges to a limit as $y\to\infty$. The
  reciprocal of this limit is called the asymptotic speed of the
  conditioned walk. We study the behavior of the asymptotic speed in
  two regimes: (1) as $p\to 0$ for $M$ fixed (``sparse''), and (2) as
  $M\to \infty$ for $p$ fixed (``spiky''). We observe and quantify a
  dramatic difference between the quenched and annealed settings.
\end{abstract}
\maketitle

\section{Introduction}
We shall start with a model description and the necessary
notation. After stating our results we offer an informal discussion,
references, and describe a conjecture which makes our results a part
of a more general picture.

\subsection{Model description and main results}
Let $V(x,\omega),\ x\in\IZ$, be i.i.d.\ random variables on a
probability space $(\Omega,{\cal F},\IP)$ such that
\begin{equation}
  \label{pot}
  \IP(V(0,\cdot)\ge 0)=1,\quad \IP(V(0,\cdot)>0)>0,\quad\text{and}\quad \IE[V(0,\cdot)]<\infty.
\end{equation}
These random variables represent a random potential on $\IZ$. 
Given a realization of the potential, i.e.\ for each fixed
$\omega\in\Omega$, we consider a Markov chain on
$\IZ\cup\{\dagger\}$ with transition probabilities:
\begin{equation}
  \label{tp}
p(x,y,\omega)
  =
  \begin{cases}
    1-e^{-V(x,\omega)},&\text{if } y=\dagger,\ x\ne\dagger;\\
    \dfrac12\, e^{-V(x,\omega)},&\text{if } y=x\pm 1,\ x\ne\dagger;\\
    1,&\text{if } y=\dagger,\ x=\dagger;\\
    0,&\text{otherwise}.\\
  \end{cases}
\end{equation}
Informally, this is a so called ``killed random walk'': at each site
$x$ the walk either gets killed with probability $1-e^{-V(x,\omega)}$
(and moves to the absorbing state $\dagger$) or survives and moves to
one of the two neighboring sites with equal probabilities
$e^{-V(x,\omega)}/2$. The corresponding path measure and the
expectation with respect to it will be denoted by ${P}^\omega$
and ${E}^\omega$ respectively. Unless stated otherwise, the
killed random walk starts from $0$.

Denote by $(S_n)_{n\ge0}$ a path on $\IZ\cup\{\dagger\}$ and set
\begin{equation}
  \label{taux}
  \tau_x:=\inf\{n>0:\,S_n=x\}\in \IN\cup\{\infty\}.
\end{equation}
We shall put different measures on the set of nearest neighbor
paths. Measure $P$, and the corresponding expectation $E$, will always
refer to the simple symmetric random walk on $\IZ$ starting from $0$.

Next, fix $y>0$ and consider the conditional measure
on nearest neighbor paths starting from $0$, which
is defined by
\begin{equation}
  \label{Qq}
  Q_y^{\omega}(\,\cdot\,):={P}^{\omega}(\,\cdot\,|\,\tau_y<\infty).
\end{equation}
Measure $Q_y^{\omega}$ is called the quenched path measure. Notice
that it can be equivalently defined as follows:
\begin{align}
  \label{qpm}
  Q^{\omega}_y(A):&=(Z^{\omega}_y)^{-1}E\left(\I_{\{\tau_y<\infty\}}
    \I_A\,e^{-\sum_{n=0}^{\tau_y-1}V(S_n,\omega)}\right),\ \text{where
  }
  \\\label{qnf} 
Z^{\omega}_y&:={P}^\omega(\tau_y<\infty)=E\left(\I_{\{\tau_y<\infty\}}
    e^{-\sum_{n=0}^{\tau_y-1}V(S_n,\omega)}\right).
\end{align}

Finally, we define the annealed measure $Q_y$, which is a measure on
the product space of
$\Omega$ and the nearest-neighbor paths starting from $0$. We let
\begin{equation}
    \label{am}
    Q_y(B):=\dfrac{\IE({P}^\omega 
(B\cap\{\tau_y<\infty\}))}{\IE({P}^\omega(\tau_y<\infty))}.
\end{equation}
Whenever the starting point of a process is different from $0$, we
shall indicate it with a superscript, for example, $Q^{\omega,x}_y$
will denote a quenched path measure of a killed random walk, which
starts at $x$ and is conditioned to hit $y$, $y>x$.
  
The quenched and annealed asymptotic speeds, $v^{\qu}_V$ and
$v^{\an}_V$, are the deterministic quantities, defined by
\begin{equation}
  \label{speeds}
  \frac{1}{v^{\qu}_V}:=\lim_{y\to\infty}\frac{E_{Q^{\omega}_y}(\tau_y)}{y}\quad
  (\IP\text{-a.s.});\quad
  \frac{1}{v^{\an}_V}:=\lim_{y\to\infty}\frac{E_{Q_y}(\tau_y)}{y}.
\end{equation}
By Proposition~\ref{basic1}
below and \cite[Theorem 1.2]{KM12} respectively these limits exist
and are finite.

In this paper we consider i.i.d.\  Bernoulli potentials
$V_{p,M}(x,\cdot)$, $x\in\IZ$, $p\in(0,1]$, $M>0$,
\begin{equation}
  \label{bp}
  \IP(V_{p,M}(x,\cdot)=M)=1-\IP(V_{p,M}(x,\cdot)=0)=p,
\end{equation}
and study the behavior of the corresponding quenched and annealed
asymptotic speeds $v^\qu_{p,M}$ and $v^\an_{p,M}$ in two regimes:
``sparse'' ($p\to 0$, $M$ is fixed) and ``spiky'' ($M\to\infty$, $p$
is fixed).  Our main results are contained in the following two
theorems. By $f(x)\sim g(x)$ as $x\to\alpha$ we mean that
$\lim_{x\to \alpha}f(x)/g(x)=1$.
\begin{theorem}
\label{vque}
With $V_{p,M}$ as above, the quenched speed, $v^{que}_{p,M} $, satisfies
\begin{align}
  \label{qp}&v_{p,M}^{\qu}   \sim \frac{3p}{2} \text{ as } p\to
  0\ \text{uniformly in } M\in [M_0,\infty),\ \forall M_0>0;\\\label{qM}
&v_{p,M}^{\qu} \sim \frac{3p}{2-p+2p^2} \text{ as } M\to\infty\
\text{locally uniformly in }p\in(0,1].
\end{align}
\end{theorem}

\begin{theorem}
\label{vann}
With $V_{p,M}$ as above, the annealed speed, $v^{ann}_{p,M} $, satisfies
\begin{align}
  \label{ap}&-\log v_{p,M}^{\an}\sim \frac{2(e^M -1)}{p} \text{ as }
  p\to 0\ \text{uniformly in } M\in[M_0,\infty),\ \forall 
  M_0>0;\\\label{aM}&-\log v_{p,M}^{\an}\sim \frac{2(1-p)e^M}{p}
  \text{ as } M\to\infty\ \text{locally uniformly in }p\in(0,1).
\end{align}
\end{theorem}
\begin{remark} 
  {\em Observe that in the ``sparse'' regime  both the
    quenched and annealed speeds vanish but the latter does so at a
    dramatically higher rate. In the ``spiky'' regime 
    the quenched speed converges to a positive constant while the
    annealed one vanishes extremely fast. This striking difference is
    a purely one-dimensional phenomenon. In dimensions two and higher
    the existence of the annealed asymptotic velocity is known
    (\cite[Theorem C]{IV12b}) but the existence of the quenched
    asymptotic velocity is still a largely open problem (see
    \cite[Section 1.2]{IV12a} and references therein). Thus the
    comparison question might seem a bit premature. Nevertheless,
    even when both speeds are well-defined, we do not expect to see
    anything like this in higher dimensions.}\label{rem}
\end{remark}
\begin{remark}
  {\em Our results can also be interpreted in terms of a closely related
    model of killed biased random walks conditioned to
    survive up to time $n$. The latter model exhibits a first order
    phase transition (in all dimensions) as the size of the bias
    increases (see, for example, \cite{MG02}, \cite{Fl07},
    \cite{IV12b}, and references therein). In dimension 1, there is a
    critical bias $b_{\ct}^{\an}$ ($b_{\ct}^{\qu}$) such that these
    random walks have the zero asymptotic speed when the bias is less than
    $b_{\ct}^{\an}$ (resp., $b_{\ct}^{\qu}$) and a strictly positive
    asymptotic  speed when the bias is greater or
    equal to $b_{\ct}^{\an}$ (resp., $b_{\ct}^{\qu}$). The model of
    crossing random walks considered in the present article informally
    corresponds to the critical bias case. In particular,
    (\ref{aM}) describes how the first order transition gap closes in
    on the second order transition in the pure trap model
    ($M=\infty$).}
\end{remark}

\subsection{Discussion and an open problem.}  There are many papers
concerning the relationship of the quenched and annealed Lyapunov
exponents of a random walk in a random potential on $\IZ^d$, $d\ge 1$
(see \cite{Fl08}, \cite{Zy09}, \cite{KMZ11}, \cite{IV12c},
\cite{Zy12}, and references therein). Lyapunov exponents represent the
exponential decay rates of the quenched and annealed survival
probabilities, ${P}^\omega(\tau_y<\infty)$ and
$\IE({P}^\omega(\tau_y<\infty))$ respectively. The equality or
non-equality of quenched and annealed Lyapunov exponents determines
whether the disorder introduced by the random environment is ``weak''
or ``strong''.  According to this classification, any non-trivial
disorder in the one-dimensional case is strong (as well as in
dimensions 2 and 3 under mild additional conditions on $V$, see
Theorem 1 and a paragraph after it in \cite{Zy12}). Our results
provide more refined information about differences between quenched
and annealed behavior.

In dimensions 4 and higher one expects a transition from weak to
strong disorder for i.i.d.\ potentials of the form $\gamma V$, $V\ge
0$, for some $\gamma^*\in(0,\infty)$.  Such result is known to hold
when $V$ is bounded away from $0$ (\cite{Fl08}, \cite{Zy09},
\cite{IV12c}).

Let us discuss the case of a ``small'' potential, i.e.\ the potential
of the form $\gamma V$ where $\gamma\ll 1$, in more detail. We shall
restrict ourselves to dimension 1 but we believe that a similar result
holds in all dimensions (when the quenched speed is well defined, see
Remark~\ref{rem} above).

Let $(V(x,\cdot))$, $x\in\IZ$, be i.i.d.\ random variables satisfying
(\ref{pot}) and $v^{\qu}_{V}$ and $v^{\an}_{V}$ be the quenched and
annealed speeds as defined in (\ref{speeds}). Using our methods it
should be not difficult to show that as $\gamma\downarrow 0$
\begin{equation}
  \label{op1}
  v^{\qu}_{\gamma V},\ v^{\an}_{\gamma V}\sim
\sqrt{2\gamma\IE(V)}.
\end{equation}
This result would complement the results of this paper. The relation
(\ref{op1}) is suggested by the following two facts.

(i) The speeds can be equivalently defined as follows
(Proposition~\ref{basic1} below and \cite[Theorem 1.2]{KM12}):
\begin{equation}
  \label{sle}
  \frac{1}{v^{\qu}_V}=\frac{d}{d\lambda}\alpha_{\lambda+V}(1)\Big|_{\lambda=0+},
  \quad
  \frac{1}{v^{\an}_V}=\frac{d}{d\lambda}\beta_{\lambda+V}(1)\Big|_{\lambda=0+},
\end{equation}
where for each $\lambda\ge 0$ the non-random quantities
\begin{align}
  \alpha_{\lambda+V}(1)&:=-\lim_{y\to\infty}\frac1y \log
  E\left(\I_{\{\tau_y<\infty\}}e^{-\sum_{n=0}^{\tau_y-1}(\lambda+V(S_n,\omega))}\right)
  \quad
  (\text{$\IP$-a.s.});\label{que}\\
  \beta_{\lambda+V}(1)&:=-\lim_{y\to\infty}\frac1y
  \log\IE\left[E\left(\I_{\{\tau_y<\infty\}}e^{-\sum_{n=0}^{\tau_y-1}
        (\lambda+V(S_n,\omega))}\right)\right]\label{ann}
\end{align}
are the quenched and annealed (respectively) Lyapunov exponents of a
random walk in the potential $\lambda+V$. Under our assumptions, both
limits are positive and finite. For more details about the existence
and properties of Lyapunov exponents see \cite{Ze98}, \cite{Fl07}, and
\cite{Mou12}.

(ii) It was shown in \cite{KMZ11} (see also \cite{Wa01}, \cite{Wa02})
that $\alpha_{\gamma V}(1),\ \beta_{\gamma V}(1)\sim
\sqrt{2\gamma\IE(V)}$ as $\gamma\downarrow 0$.  Since
\[\alpha_{\gamma(\lambda+V)}(1)\sim\sqrt{2\gamma(\lambda+\IE(V))} 
\quad\text{as $\gamma\downarrow 0$ \quad and}\quad \frac{d}{d(\gamma
  \lambda)}=\frac1\gamma \frac{d}{d\lambda},\] a formal
differentiation of $\sqrt{2\gamma(\lambda+\IE(V))}$ with respect to
$\lambda$ and division by $\gamma$ lead to the conjecture
(\ref{op1}). We remark that in the case of a constant potential
(without loss of generality we set $V\equiv 1$) it is easy to compute
the Lyapunov exponent explicitly:
$\alpha_{\lambda+\gamma}(1)=\beta_{\lambda+\gamma}(1)=\log(e^{\lambda+\gamma}
+\sqrt{e^{2(\lambda+\gamma)}-1})$, $\gamma,\lambda\ge 0$, and obtain
from (\ref{sle})
\[v_\gamma:=v^{\qu}_{\gamma}=v^{\an}_{\gamma}=\sqrt{1-e^{-2\gamma}},\quad
\gamma\ge 0.\] Thus, $\alpha_{\gamma}(1)=\beta_{\gamma}(1)\sim
\sqrt{2\gamma}$ and $v_{\gamma}\sim \sqrt{2\gamma}$ as
$\gamma\downarrow 0$. The latter is (\ref{op1}) in this special case.

The relation (\ref{op1}) further supports an informal statement that when
the potential $V$ is ``small'' (but not ``sparse''!) both the quenched
and annealed behavior in such a potential are well approximated by the
behavior of the walk in a constant potential $\IE(V)$.

\textit{An open problem.} Theorems~\ref{vque}, \ref{vann}, and
asymptotics (\ref{op1}) provide information about $v^\qu_{p,M}$ and
$v^\an_{p,M}$ when $M\to\infty$ or $p\to 0$ or $M\to 0$. A more
interesting and challenging question is to study surfaces formed by
the speeds when $(p,M)\in[0,1]\times[0,\infty)$. 

As the first step, fix $p\in(0,1)$ and consider $v^\qu_{p,M}$ and
$v^\an_{p,M}$ as functions of $M$. It does not seem surprising that
the annealed environment will typically be more sparse than the
quenched one and thus the random walker will feel less need to quickly
navigate towards the goal $y$. One expects that $v^\qu_{p,M}$ is
strictly increasing in $M$ for a fixed $p$ and that $v^\an_{p,M}$ has
a single maximum as $M$ changes from $0$ to $\infty$ (see
Figure~\ref{fig}).
\begin{figure}[h]
  \centering
  \includegraphics[height=7cm]{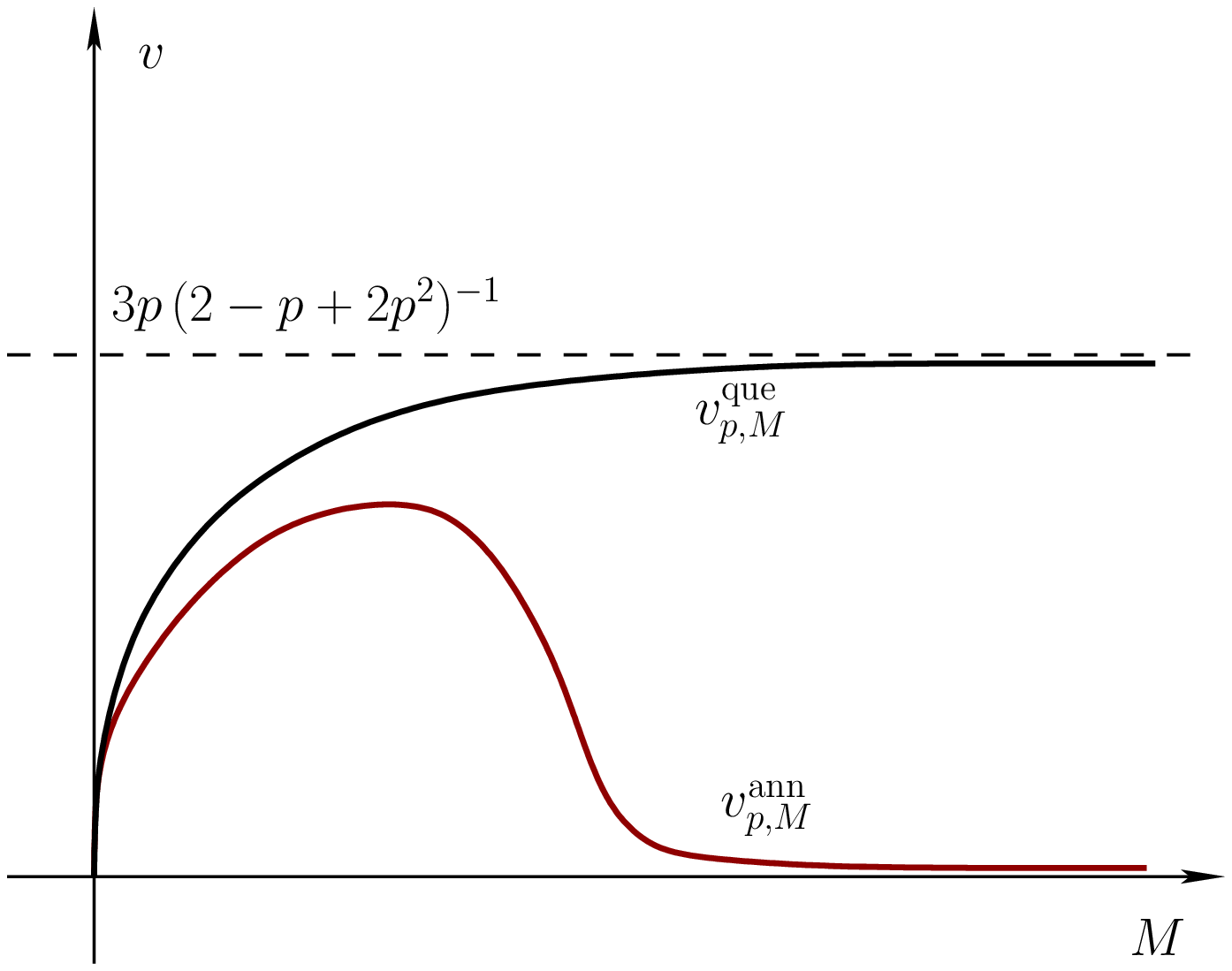}
  \caption{Conjectured shape of $v^\qu_{p,M}$ and $v^\an_{p,M}$ as 
    functions of $M$ for a fixed $p\in(0,1)$. By (\ref{op1}),
    $v^\qu_{p,M},v^\an_{p,M}\sim \sqrt{2pM}$ as $M\to 0$. }
  \label{fig}
\end{figure}
Is this really the case? These questions are presumably hard
to resolve analytically. We are not aware even of any simulations done
on this problem (see though \cite{MG02} as well as \cite{GP82}).

\subsection{Organization of the paper.} In Section ~2 we prove
Theorem~\ref{vque}. The proof of Theorem~\ref{vann} is given in
Section~3, which is subdivided into 3 subsections. Subsection 3.1
gives heuristics and an outline of the proof. Subsection 3.2 is the
technical core of the proof. There, after providing an informal
calculation, we study the environment under the annealed
measure. Subsection 3.3 uses the estimates obtained in the previous
subsection and completes the proof of Theorem~\ref{vann}. The Appendix
contains several auxiliary results and proofs of several lemmas used
in the main part of the paper.

\subsection{Terminology.} Sites at which the potential is equal to $M$
will be called {\it occupied } sites or {\em obstacles}. {\it Vacant}
sites are unoccupied sites.  An interval is {\it empty} if all its
sites are unoccupied.  We reserve the term {\em vacant interval} for maximal
empty intervals. A {\em gap} between two occupied sites is the length
of the vacant interval between these two sites.

\section{Quenched speed}

The results of Theorem~\ref{vque} are rather straightforward.  The
fact that the speed is at most of order $p$ comes immediately from the
fact that no matter what the value of $M$, the time for a conditioned
random walk to traverse an empty interval of size $R$ will be of order
$R^2$, corresponding to speed of order $1/R$.  The upper bound on
speed follows since most sites lie in vacant intervals of size of
order $1/p$. 

We shall need two facts.  The first is a very basic fact
about the standard random walk but we do not have a reference at hand
and, thus, give a proof in the Appendix.
\begin{proposition} \label{basic1} Let $(S_j)_{j\ge 0}$ be the simple
  symmetric random walk, $n\in\IN$, and
  $S_0=k\in\{1,2,\dots,n-1\}$. Then 
\[E^k(\tau_n;\tau_n<\tau_0)=\frac{k(n-k)(n+k)}{3n}.\]
\end{proposition}
The second is mostly a consequence of the ergodic theorem (see
\cite[(1.30) and Theorem 2.6]{Sz94} for a treatment of Brownian motion
among Poissonian obstacles). The proof is given in the Appendix. 
\begin{proposition} \label{basic2} Let $V(x,\cdot)$, $x\in\IZ$ be
  i.i.d.\ random variables, which satisfy (\ref{pot}). Then there exists
  limit
  \begin{align*}
    \frac{1}{v^\qu}&:=\lim_{y\to\infty}\frac{E_{Q^\omega_y}(\tau_y)}{y}
    =\IE(E^\omega(\tau_1\,|\,\tau_1<\infty))<\infty
    \quad(\IP\text{-a.s.}),\quad \text{and, moreover,}\\
    \frac{1}{v^\qu}&=\frac{d}{d\lambda}\alpha_{\lambda+V}(1)\Big|_{\lambda=0+}.
  \end{align*}
\end{proposition}

The key idea for calculation of the quenched speed is an observation that
the main contribution to $\IE[
E^\omega(\tau_1\,|\,\tau_1<\infty)]$ comes from paths which hit
$1$ before entering $(-\infty,-a_1]$ where $-a_1=\max\{x\le 0:
V(x,\cdot)=M\}$ at some positive time.  Our first step is to compute the
main term.

\begin{lemma} \label{escape} For every $a_1\in\IN\cup\{0\}$ and $p\in(0,1]$\[ \IE[
  E^\omega(\tau_1\,|\,\tau_1<\tau_{-a_1})] = \frac{2-p+2p^2}{3p}.
\]
\end{lemma} 
\begin{proof}
  Notice that our definition of $\tau_0$ (see (\ref{taux})) implies
  $E^\omega(\tau_1\,|\,\tau_1<\tau_0)=1$. By Proposition~\ref{basic1}
  we have
  \begin{align*}
    \IE [E^\omega(\tau_1\,|\,\tau_1<\tau_{-a_1})]=
    p+p\sum_{a=1}^\infty
    \frac{2a_1+1}{3}\,(1-p)^{a_1}=
    p+\frac{1-p}{3}\left(\frac{2}{p}+1\right)= \frac{2-p+2p^2}{3p}.
  \end{align*}
\end{proof}
Given $\omega$, let $-a_1>-a_2>\dots$ be the occupied sites in
$(-\infty,0]$, $I_j=[-a_{j+1},-a_j)$, $j\in\IN$, and
$I=[-a_1,1)$. Then
\begin{multline}\label{dec1}
  \IE[ E^\omega(\tau_1\,|\,\tau_1<\infty)]=\IE[
  E^\omega(\tau_1\I_{\{\tau_1<\tau_{-a_1}\}}\,|\,\tau_1<\infty)]
  \\ +\IE[
  E^\omega(\I_{\{\tau_1>\tau_{-a_1}\}}\sum_{n=0}^{\tau_1-1}\I_{\{S_n\in
    I\}}\,|\,\tau_1<\infty)]+ \sum_{j=1}^\infty\IE[
  E^\omega(\I_{\{\tau_1>\tau_{-a_1}\}}\sum_{n=0}^{\tau_1-1}\I_{\{S_n\in
    I_j\}}\,|\,\tau_1<\infty)].
\end{multline}
Observe also that the first term in the right hand side of (\ref{dec1})
equals
\begin{multline}\label{dec2}
  \IE[ E^\omega(\tau_1\,|\,\tau_1<\tau_{-a_1})
  P^{\omega}(\tau_1<\tau_{-a_1}\,|\,\tau_1<\infty)]=
  \\\IE[E^\omega(\tau_1\,|\,\tau_1<\tau_{-a_1})]-
  \IE[E^\omega(\tau_1\,|\,\tau_1<\tau_{-a_1})P^{\omega}(\tau_1>\tau_{-a_1}\,|\,\tau_1<\infty)].
\end{multline}
We shall need the following three elementary lemmas.
\begin{lemma}\label{elem1} For $a_1\in\IN\cup\{0\}$
  \begin{equation*}
    P^{\omega}(\tau_1>\tau_{-a_1}\,|\,\tau_1<\infty)\le
    e^{-M}\wedge ((2a_1(e^M-1)+1)(1+a_1))^{-1}.
  \end{equation*}
\end{lemma}

\begin{lemma} \label{elem2} There is a constant $C_1$ such that
  \begin{equation*}
    E^\omega(\I_{\{\tau_1>\tau_{-a_1}\}}\sum_{n=0}^{\tau_1-1}\I_{\{S_n\in
    I\}}\,|\,\tau_1<\infty)\le C_1(e^M-1)^{-1}.
  \end{equation*}
\end{lemma}

\begin{lemma} \label{elem3} There is a constant $C_2$ such that for
  every $j\in\IN$
\begin{equation*}
  E^\omega(\I_{\{\tau_1>\tau_{-a_1}\}}\sum_{n=0}^{\tau_1-1}\I_{\{S_n\in
    I_j\}}\,|\,\tau_1<\infty)\le \frac{C_2e^{-Mj}} 
  {(1-e^{-M})^2}\frac{|I_j|}{|I|^2},
\end{equation*}
where $|A|$ is the Lebesgue measure of the set $A$.
\end{lemma}
Let us assume these facts (see Appendix for proofs) and derive
Theorem~\ref{vque}.
\begin{proof}[Proof of Theorem~\ref{vque}]
  The right hand side of the inequality in Lemma~\ref{elem1} does not
  exceed $e^{-M}/(1+a_1)$. Therefore, 
  \begin{multline*}
    \IE[E^\omega(\tau_1\,|\,\tau_1<\tau_{-a_1}) 
P^{\omega}(\tau_1>\tau_{-a_1}\,|\,\tau_1<\infty)]\le\\
    e^{-M}+pe^{-M}\sum_{a_1=1}^\infty
    (1-p)^{a_1}\frac{2a_1+1}{3(a_1+1)}<
    e^{-M}+\frac{2pe^{-M}}{3}\sum_{a_1=1}^\infty (1-p)^{a_1}\le
    \frac{5}{3}\,e^{-M}.
  \end{multline*}
This immediately gives 
  \begin{align*}
    &\lim_{p\to
      0}p\,\IE[E^\omega(\tau_1\,|\,\tau_1<\tau_{-a_1})
    P^{\omega}(\tau_1>\tau_{-a_1}\,|\,\tau_1<\infty)]=0;\\
    &\lim_{M\to
      \infty}\IE[E^\omega(\tau_1\,|\,\tau_1<\tau_{-a_1}) 
    P^{\omega}(\tau_1>\tau_{-a_1}\,|\,\tau_1<\infty)]=0.
  \end{align*}
  Lemma~\ref{elem2} takes care of the second term in the right hand
  side of (\ref{dec1}). By Lemma~\ref{elem3} and the independence of
  the values of the potential at distinct sites, the last term in
  (\ref{dec1}) is bounded by (we defined the function $(1-p)^{-1}\ln(1/p)$ to be
$1$ at $p=1$ by continuity)
\begin{multline*}
  \frac{C_2}{(1-e^{-M})^2}\sum_{j=1}^\infty
  e^{-Mj}\IE[(a_{j+1}-a_j)]\IE[(a_1+1)^{-2}]\\\le\frac{C_2}
  {(1-e^{-M})^2}\sum_{j=1}^\infty e^{-Mj}\,\frac{\ln
    (1/p)}{1-p}=\frac{C_2e^{-M}\ln
   (1/p)}{(1-p)(1-e^{-M})^3}.
\end{multline*}

This expression clearly vanishes as $M\to \infty$ locally uniformly in
$p\in(0,1]$. After multiplication by $p$ it converges to $0$ as $p\to
0$ uniformly on every interval $[M_0,\infty)$. The only term left in
  the right hand side of (\ref{dec1}) and (\ref{dec2}) is the main
  term, $\IE[ E^\omega(\tau_1\,|\,\tau_1<\tau_{-a_1})]$, which has the
  claimed asymptotics by Lemma~\ref{escape}.
\end{proof}

\section{Annealed speed}
We start by introducing additional notation.
Let \[Q_{0,y}(\cdot):=Z_{0,y}^{-1}\IE[P^\omega (\ \cdot\
;\tau_y<\tau_0,\tau_y<\infty)],\quad
Z_{0,y}:=\IE[P^\omega(\tau_y<\tau_0,\tau_y<\infty)].\] The corresponding
quenched path measure $Q_{0,y}^\omega$ is given
by \[Q_{0,y}^\omega(\cdot)=(Z^\omega_{0,y})^{-1} {P}^\omega(\ \cdot\
;\tau_y<\tau_0,\tau_y<\infty),\quad Z^\omega_{0,y}:=
{P}^\omega(\tau_y<\tau_0,\tau_y<\infty). \]

\subsection{Heuristics and goals.} 
Our first observation is that we can replace the measure $Q_y$ with
$Q_{0,y}$ (see Proposition 3.1, (3.4), and the proof of (1.7) in
\cite{KM12}).  The key ingredient of the proof of Theorem~\ref{vann}
is the study of environments under $Q_{0,y}$. We show that under
$Q_{0,y}$ the distribution of gaps between occupied sites is
``comparable'' to a product of log-series
distributions\footnote{Random variable $R$ is said to have a log-series
  distribution with parameter $p$ if $P(R=r)=C_pp^r/r$, $r\in\IN$,
  where $C_p=-(\ln(1-p))^{-1}$.}  with the average gap $g_\an$,
where for fixed $M$ and small $p$ or fixed $p$ and large $M$
\begin{equation}
\label{gap} 
\log g_\an\sim K(p,M):=2p^{-1}(1-p)(e^M-1).
\end{equation}
For i.i.d.\ Bernoulli potentials the gap distribution is geometric,
and we already have the result that the reciprocal of the quenched
speed is proportional to the average gap between two occupied sites,
which is now $g_\an$. This observation together with (\ref{gap}) leads
to the limits (\ref{ap}) and (\ref{aM}).

We shall give a detailed proof of (\ref{ap}). The proof of (\ref{aM})
is very similar but easier and is omitted but we shall write all steps
in such a way that they can be readily adapted to the case when
$M\to\infty$ and $p$ is fixed. An informal derivation of the formula
for $K(p,M)$ is given in the next subsection right after
Corollary~\ref{sp}.

Our goal will be to construct subsets of environments that are
essential and on which the walk has the claimed 
speed behavior. More precisely, to obtain a lower bound on $-\log
v^{ann}_{p,M}$ we shall restrict $Q_{0,y}$ to environments
$\Omega^1_y=\Omega^1_y(\epsilon,p,M)$ with the following properties: for every
$\epsilon>0$ there is $p_0=p_0(\epsilon)$ such that for each $p<p_0$
there is $y_0=y_0(p)$ such that for all $y>y_0$
\begin{itemize}
\item [(L1)] 
    $Q_{0,y}(\Omega^1_y)\ge 1/2$ and
\item [(L2)] for every $\omega\in\Omega^1_y$ 
  \begin{equation*}
    E_{Q_{0,y}^\omega}(\tau_y)\ge
    C_1ye^{(1-\epsilon)K(p,M)},
  \end{equation*}
\end{itemize}
where $C_1$ does not depend on $y,\omega,p$. Then
\begin{multline*}
\frac{E_{Q_{0,y}}(\tau_y)}{y}\ge
(yZ_{0,y})^{-1}\IE(E_{Q^\omega_{0,y}}(\tau_y)Z^\omega_{0,y};
\Omega^1_y) \ge C_1e^{(1-\epsilon)K(p,M)}\,
Q_{0,y}(\Omega^1_y)\ge \frac{C_1}{2}\,e^{(1-\epsilon)K(p,M)},
\end{multline*}
and, hence,
\begin{equation}
  \label{lb}
  -\liminf_{p\to 0}p\log v^{ann}_{p,M}\ge
(1-\epsilon)\lim_{p\to 0}pK(p,M)=2(1-\epsilon)(e^M-1).
\end{equation}
For an upper bound we shall consider environments
$\Omega^2_y=\Omega^2_y(\epsilon,p,M)$ for which the following holds:
for every $\epsilon>0$ there is $p_0=p_0(\epsilon)$ such that for each
$p<p_0$
\begin{itemize}
\item [(U1)] $\lim\limits_{y\to\infty}y\,Q_{0,y}(\Omega\setminus
  \Omega^2_y)=0$ and
\item [(U2)] there is $y_0=y_0(p)$ such that for all $y>y_0$,
  $\omega\in \Omega^2_y$
\[E_{Q_{0,y}^\omega}(\tau_y)\le C_2y
  e^{(1+\epsilon)K(p,M)},\]
\end{itemize}
where $C_2$ does not depend on $y,\omega,p$. Then
\begin{multline*}
\frac{E_{Q_{0,y}}(\tau_y)}{y}\le
(yZ_{0,y})^{-1}\left(\IE(E_{Q^\omega_{0,y}}(\tau_y)Z^\omega_{0,y};
\Omega^2_y)
+\IE(E_{Q^\omega_{0,y}}(\tau_y)Z^\omega_{0,y};
\Omega\setminus\Omega^2_y)\right)\\ \le
C_2e^{(1+\epsilon)K(p,M)}+(yZ_{0,y})^{-1}\IE(E_{Q^\omega_{0,y}}(\tau_y)Z^\omega_{0,y};
\Omega\setminus\Omega^2_y).
\end{multline*}
By Lemma~\ref{A} (see Appendix), 
$E_{Q^\omega_{0,y}}(\tau_y)\le 3y^2$. Combining this with (U1)
we get
\begin{equation}
  \label{ub}
  -\limsup_{p\to 0}p\log v^{ann}_{p,M}\le
(1+\epsilon)\lim_{p\to 0}pK(p,M)=2(1+\epsilon)(e^M-1).
\end{equation}
Since $\epsilon$ is arbitrary, relations (\ref{lb}) and (\ref{ub})
imply (\ref{ap}). Our task will be to construct $\Omega^i_y$, $i=1,2$,
with the desired properties. The starting point for obtaining (L2) and
(U2) is Lemma~\ref{U} which gives bounds on
$E_{Q^\omega_{0,y}}(\tau_y)$ in terms of gaps between obstacles.  The
construction of $\Omega^i_y$, $i=1,2$, will be carried out in
Subsection~\ref{last} after we obtain information about a typical
environment under the annealed measure $Q_{0,y}$. The latter is the
content of the next subsection.

\subsection{Environment under the annealed measure}
 
\begin{lemma}
  \label{F} Let $x_0:=0<x_1<\dots<x_n$, $r_i=x_i-x_{i-1}$,
  $i=1,2,\dots,n$, and consider an environment such that
  $\{x_1,x_2,\dots,x_n\}$, $n\in\IN$, is the set of all occupied sites
  in $(0,x_n]$.  Denote by $u_n$ the probability that a random walk
  starting at $x_{n-1}$ reaches $x_n$ before hitting $0$, i.e.\
  $u_n={P}^{x_{n-1},\omega}(\tau_{x_n}<\tau_0)$, $n\in\IN$. Then
  \begin{equation}
    \label{u_n}
    u_1=\frac{e^{-V(0,\omega)}}{2r_1}=e^{M-V(0,\omega)}F_M(0,r_1,0); 
\  u_n=F_M(r_{n-1},r_n,u_{n-1}),\ n>1,
\end{equation}
where 
$F_M:(\IN\cup\{0\})\times\IN\times[0,1]\to [0,1]$,
\begin{equation}\label{defF}
    F_M(\ell,r,u)=
    \begin{cases}
      \dfrac{e^{-M}}{2r}\,\left(1- e^{- M}\left(1- \dfrac{1}{2 r}
        - \dfrac{1-u}{2 \ell} \right)\right)^{-1},&\text{if } \ell\ne 0;\\[3mm]
    \dfrac{e^{-M}}{2r},&\text{if } \ell=0.
    \end{cases}
\end{equation}
\end{lemma}
The proof is given in the Appendix. 

\begin{lemma} As $\ell,r\to\infty$
  \begin{equation}
    \label{Fb}
    F_M(\ell,r,u)\sim \frac{e^{- M}}{2 r (1- e^{- M})},
  \quad\text{uniformly in $u\in[0,1]$.}
  \end{equation}
\end{lemma}
From now on we shall identify every environment $\omega$ on $(0,y)$
with the vector $\overline{R}_N=(R_1,R_2,\dots,R_N)\in
\cup_{i=1}^\infty\IN^i$ of $N=N(y,\omega)$ successive distances
between occupied sites in $(0,y)$, where $R_1$ is the
distance from the first positive occupied site in $(0,y)$ to the
origin and $R_N$ is the distance from the last occupied site in
$(0,y)$ to $y$. If the interval $(0,y)$ is empty then we set $N=1$ and
$R_1=y$.
\begin{corollary}\label{sp} For any $n\in\{1,2,\dots,y\}$ and
  $(r_1,r_2,\dots,r_n)\in\IN^n$ with $\sum_{i=1}^n r_i=y$,
  $r_0=0$, $u_0=0$,
  \begin{align}
\label{23}
Q_{0,y}&(N=n, \overline{R}_n=(r_1,r_2,\dots,r_n))\nonumber
\\=&Z_{0,y}^{-1}\left(\frac{e^M(1-p)}{p}+1
\right)\prod^n_{i=1} \left(p(1-p)^{r_i-1}\,F_M(r_{i-1}, r_{i},
  u_{i-1})\right)\\=&Z_{0,y}^{-1}(1-p)^y\left(\frac{e^M}{\rho}+1
\right)\prod^n_{i=1} \left(\rho\,F_M(r_{i-1}, r_{i},
  u_{i-1})\right),\ \text{where $\rho:=p/(1-p)$.}\nonumber
\end{align}
\end{corollary}

\noindent\textbf{Heuristic derivation of (\ref{gap}).} Before we
    turn to  rigorous analysis of (\ref{23}) we would like to present
    a ``back of the envelope derivation'' of the gap asymptotics
    (\ref{gap}). When $y\to\infty$ we might expect that measures
    $Q_{0,y}$ converge to a limiting measure, under which the
    consecutive gaps are essentially i.i.d.. It is reasonable to
    assume that if we let $M\to\infty$ or $p\to 0$ then the distances
    between consecutive occupied sites under this limiting measure
    will also go to infinity.  Thus, we replace $F_M(r_{i-1}, r_{i},
    u_{i-1})$ in (\ref{23}) with its limit as $r_{i-1}, r_{i}\to
    \infty$ given by (\ref{Fb}).  We get that for $y\to \infty$
  \begin{equation}\label{eu2}
    Q_{0,y}(N=n, \overline{R}_n=(r_1,r_2,\dots,r_n)) \ \asymp\ 
Z_{0,y}^{-1}\prod^n_{i=1} \left(p(1-p)^{r_i-1}\,\frac{e^{- M}}{2 r_i
    (1- e^{- M})}\right).
  \end{equation}
  By \cite[Lemma 5.5]{KM12}, $\lim_{y\to\infty}y^{-1}\ln
  Z_{0,y}=\beta$, where $\beta:=\beta_V(1)$ is the annealed Lyapunov
  exponent (see (\ref{ann})), and we replace $Z_{0,y}^{-1}$ in
  (\ref{eu2}) with
  $e^{\beta y}=e^{\beta\sum_{i=1}^n r_i}$ to arrive at
  \begin{multline*}
    Q_{0,y}(N=n, \overline{R}_n=(r_1,r_2,\dots,r_n)) \ \asymp\  
    \prod^n_{i=1} \left(\frac{p}{2 (1-p)(e^M-
      1)}\, \frac{(e^{\beta}(1-p))^{r_i}}{r_i}\right) \\= \prod^n_{i=1}
  \left(\frac{1}{K} \frac{(e^{\beta}(1-p))^{r_i}}{r_i}\right),
  \end{multline*}
where
  $K=K(p,M)$ is the same as in (\ref{gap}). For $Q_{0,y}$ to be a
  probability measure it should hold that
  \begin{equation}
    \label{eu3}
    \frac1{K}\sum_{r=1}^\infty\frac{(e^{\beta}(1-p))^r}{r}=1.
  \end{equation}
  In other words, the limiting gap size appears to have the
  so-called log-series distribution. Summing up the series in
  (\ref{eu3}) we see that
  $\ln(1-e^{\beta}(1-p))=-K$, i.e. $e^\beta(1-p)=1-e^{-K}$, 
and conclude that the expected gap size 
is \[\frac{1}{K}\sum_{r=1}^\infty(e^\beta(1-p))^r=\frac{1}{K}\,\sum_{r=1}^\infty(1-e^{-K})^r=\frac{e^K-1}{K}.\]
This immediately leads to (\ref{gap}). 

\medskip

Here is a layout of the rest of this subsection. We start a
rigorous analysis by noticing that all information about the
dependence of (\ref{23}) on $(r_1,r_2,\dots,r_n)$ is contained in the
last product. To study its behavior, we consider an auxiliary
quantity, the probability $U_n(q)$ that the killed random walk reaches
the $n$-th occupied site prior to the first return to $0$ if
$(R_i)_{i\in\IN}$, are i.i.d.\ positive integer-valued random
variables with probability mass function $G_q(r):=q(1-q)^{r-1}$,
$r\in\IN\cup\{0\}$, for some $q\in (0,1)$.  Without loss of generality
we shall assume that $0$ is occupied. Then (setting $r_0=0$, $u_0=0$)
\begin{equation}
  \label{Un}
  U_n(q)=\sum_{r_1,r_2,\dots,r_n\in\IN} \prod_{i=1}^n G_q(r_i)
F_M(r_{i-1}, r_i, u_{i-1}).
\end{equation}
We notice that $U_n(q)$ decays exponentially fast in $n$ for each $q$
(Lemma~\ref{lem5basic}). If we want the event that the killed random
walk reaches the $n$-th occupied site prior to the first return to $0$
to be a typical event, then we need to renormalize (\ref{Un}). In
Corollary~\ref{corbase} we show that there is $q=q(p)$ such that,
after the renormalization, the probability of the above event is
essentially equal to 1 (see (\ref{co1})). The renormalized measures
(\ref{pi}) can be effectively compared with product measures
(Lemma~\ref{lem54}). Such comparison allows us to use standard large
deviation bounds for product measures (Corollary~\ref{corexp}) and
obtain sufficient control on the right-hand side of (\ref{23}) to be
able to construct $\Omega^1_y$ and $\Omega^2_y$ in the next
subsection.

We use $\rho$ below simply as a shorthand for
$p/(1-p)$. Obviously $\rho\sim p$ as $p\to 0$.
\begin{lemma} \label{lem5basic} There is a continuous function
  $\mu:(0,1)\to \IR$ such that for every $q\in(0,1)$ 
  \begin{equation}
    \label{mu}
    \frac{q\log(1/q)}{2e^M(1-q)}\le \mu(q)\le  
\frac{q\log(1/q)}{2(e^M-1)(1-q)},
  \end{equation}
and $\mu^n(q)(1-e^{- M}) \le
  U_n(q) \leq \mu^n(q)$ for all $n\in\IN$.
\end{lemma}
\begin{proof} Let us fix an arbitrary $q\in(0,1)$ and drop it from the
  notation.  It is obvious that $U_n\ge U_mU_{n-m}$ for $1\le m\le
  n$. This implies that the sequence $\log U_n$, $n\in\IN$, is
  superadditive, and, thus,
  \begin{equation}
    \label{un}
    \lim_{n\to\infty} \dfrac{\log U_n}{n}=\sup_n\frac{\log
    U_n}{n}=:\log \mu.
  \end{equation}
Therefore, $U_n\le \mu^n$ for all
  $n\in\IN$. 

  For the lower bound, consider a killed random walk, which starts
  from the origin in an environment, such that all sites to the left
  from $0$ are empty. Let $\tilde{U_n}$ be the probability that this
  walk reaches the $n$-th occupied site in $(0,\infty)$. Conditioning
  on the number of returns to the origin before reaching the $n$-th
  occupied site, we obtain 
  \begin{equation}\label{us}
    U_n\le \tilde{U_n}\le
  \sum_{k=0}^\infty U_ne^{-Mk}=\frac{U_n}{1-e^{-M}}.
  \end{equation}
  Notice that the sequence $\log\tilde{U_n}$, $n\in\IN$, is
  subadditive. From this, (\ref{us}), and (\ref{un}) it follows that \[\log
  \mu= \lim_{n\to\infty} \dfrac{\log \tilde{U_n}}{n}=\inf_n\frac{\log
    \tilde{U_n}}{n}.\] We conclude that $\tilde{U_n}\ge \mu^n$. By
  (\ref{us}), $U_n \ge \mu^n(1-e^{- M}) $ for all $n\in\IN$.

  Properties of $\mu=\mu(q)$ follow from the inequality $U^{1/n}_n\le \mu\le
  U^{1/n}_n(1-e^{-M})^{-1/n}$, $n\in\IN$. Taking $n=1$ we
  compute directly that \[U_1(q)=\sum_{r=1}^\infty
  q(1-q)^{r-1}/(2e^Mr)=\frac{q\log(1/q)}{2e^M(1-q)} \] and obtain the
  desired bounds on $\mu$. Continuity of $\mu(q)$ follows from continuity
  of $U^{1/n}_n(q)$ as a function of $q$ for each $n$ and the fact that
  $(1-e^{-M})^{-1/n}\to 1$ as $n\to\infty$.
\end{proof}
\begin{corollary} \label{corbase} For each $\rho\in(0,\infty)$ there
  is a $q=q(\rho,M)\in(0,1)$ such that 
\begin{equation}
    \label{rho}
    2(1-e^{-M})\le
  e^{-M}\rho\log (1/q)\le 2,
  \end{equation} and for all $n\in\IN$
\begin{equation}
  \label{co1}
  (1-e^{-M})\le \sum_{r_1,r_2,\dots,r_n\in\IN} \prod_{i=1}^n
\left((1-q)^{r_i}\rho F_M(r_{i-1}, r_i, u_{i-1})\right)\le 1. 
\end{equation}
Moreover, for every $m\in\{1,2,\dots,n\}$ and any
$r_1,r_2,\dots,r_{m-1}\in\IN$
\begin{equation}
  \label{co2}
  (1-e^{-M})\le \sum_{r_m,r_{m+1},\dots,r_n\in\IN} \prod_{i=m}^n
\left((1-q)^{r_i}\rho F_M(r_{i-1}, r_i, u_{i-1})\right)\le (1-e^{-M})^{-1}.
\end{equation}
\end{corollary}
\begin{proof}
  Lemma~\ref{lem5basic} implies that
  \[(1-e^{- M})\le \sum_{r_1,r_2,\dots,r_n\in\IN} \prod_{i=1}^n
  (1-q)^{r_i}\dfrac{q}{(1-q)\mu} F_M(r_{i-1}, r_i, u_{i-1}) \le 1.\]
  Setting $\rho= q/(\mu(1-q))$ we get (\ref{co1}). Properties of $\mu$
  (see (\ref{mu})) imply (\ref{rho}).

  To show (\ref{co2}) we first notice that applying (\ref{co1}) to the
  walk which starts at $x_{m-1}$, never returns to $x_{m-1}$, and reaches $x_n$
  (in the notation of Lemma~\ref{F}) we get
  \begin{multline}
\label{red}
    (1-e^{-M})\le \sum_{r_m,r_{m+1},\dots,r_n\in\IN}
    \Big((1-q)^{r_m}\rho F_M(0, r_m, 0)
    \prod_{i=m+1}^n (1-q)^{r_i}\rho F_M(r_{i-1}, r_i,
        u_{i-1})\Big)\le 1.
  \end{multline}
Moreover, by (\ref{defF}),
\[F_M(0, r_m, 0)\le F_M(r_{m-1}, r_m, u_{m-1})\le\frac{F_M(0, r_m,
  0)}{1-e^{-M}}.\] Thus, we can replace $F_M(0, r_m, 0)$ with
$F_M(r_{m-1}, r_m, u_{m-1})$ in (\ref{red}) at the expense of an extra
factor in the right-hand side and obtain (\ref{co2}).
\end{proof}
From Corollary~\ref{corbase} we see that for each $\rho\in(0,\infty)$
measures $\Pi^\rho_n$ on $\IN^n$, $n\in\IN$, defined
by
\begin{equation}
  \label{pi}
  \Pi^\rho_n(\{r_1,r_2,\dots,r_n\}):=\prod_{i=1}^n (1-q)^{r_i}\rho
F_M(r_{i-1}, r_i, u_{i-1}),\ (r_1,r_2,\dots,r_n)\in\IN^n,
\end{equation}
form an ``almost'' consistent family of ``almost'' probability
measures.

We can sharpen (\ref{rho}) as follows.
\begin{proposition} \label{propcalc} Let $q$ be defined as in
  Corollary~\ref{corbase} and $\rho=p/(1-p)$. Then \[\lim_{\rho\to
    0}\rho\log(1/q)=2(e^M-1)\ \text{ and }\
  \lim_{M\to\infty}e^{-M}\log(1/q)=2/\rho.\]
\end{proposition}
\begin{proof}
  In view of (\ref{rho}) we only need to show that for every
  $\epsilon>0$ there is $\rho_0>0$ such that $\rho\log(1/q)\le
  2(1+\epsilon)(e^M-1)$ for all $\rho\in(0,\rho_0)$. Assume the
  contrary, i.e.\ that there is $\epsilon\in(0,1/2)$ such that for
  every $\rho_0>0$ there is $\rho\in(0,\rho_0)$, for which
  $\rho\log(1/q)> 2(1+\epsilon)(e^M-1)$. Then \[\sum_{r=1}^\infty
  \frac{\rho(1-q)^r}{2(e^M-1)r}>1+\epsilon.\] We shall show that this
  contradicts to the fact that the total mass of $\Pi^\rho_n$ is
  bounded above by $1$ uniformly in $n\in\IN$ and
  $\rho\in(0,\infty)$. Recall that for all $\ell\ge 0,\ r\ge 1$ and
  $u\in[0,1]$ \[F_M(\ell,r,u)\ge
  F_M(0,r,0)=\frac{1}{2re^M}=\frac{1-e^{-M}}{2r(e^M-1)}\] and that, by
  (\ref{Fb}), for all $\ell,r> R$, where $R$ is sufficiently large,
  \[F_M(\ell,r,u)\ge \frac{(1-\epsilon/4)}{2r(e^M-1)}.\] Given
  $r_1,r_2,\dots,r_n$, we split the set of indices into
  $A_R:=\{i\in\{1,2,\dots,n\}:\,\min\{r_{i-1},r_i\}\le R\}$ and
  $\bar{A}_R:=\{1,2,\dots,n\}\setminus A_R$. Then
  \begin{multline*}
    \Pi^\rho_n(\{r_1,\dots,r_n\})=\prod_{i=1}^n \rho(1-q)^{r_i}
    F_M(r_{i-1},r_i,u_{i-1})=\\\prod_{i\in A_R}\rho(1-q)^{r_i}
    F_M(r_{i-1},r_i,u_{i-1}) \times
    \prod_{i\in\bar{A}_R}\rho(1-q)^{r_i} F_M(r_{i-1},r_i,u_{i-1})\ge
    \\ (1-e^{-M})^{|A_R|}(1-\epsilon/4)^{n-|A_R|}\prod_{i=1}^n
    \frac{\rho(1-q)^{r_i}}{2(e^M-1)r_i}.
  \end{multline*}
  Choose $\delta>0$ small enough to have $(1-e^{-M})^\delta\ge
  1-\epsilon/4$. Then
  \begin{align*}
    1&\ge \Pi^\rho_n(|A_R|\le \delta n)=\sum_{|A_R|\le \delta n}
    \prod_{i=1}^n \rho(1-q)^{r_i} F_M(r_{i-1},r_i,u_{i-1})\\&\ge
    (1-\epsilon/4)^{2n}\sum_{|A_R|\le \delta n} \prod_{i=1}^n
    \frac{\rho(1-q)^{r_i}}{2(e^M-1)r_i}
    \\&=(1-\epsilon/4)^{2n}\left(\sum_{r_1,\dots,r_n\in\IN}
      \prod_{i=1}^n \frac{\rho(1-q)^{r_i}}{2(e^M-1)r_i}-\sum_{|A_R|>
        \delta n} \prod_{i=1}^n \frac{\rho(1-q)^{r_i}}{2(e^M-1)r_i}\right)\\
    &\ge
    (1-\epsilon/4)^{2n}(1+\epsilon)^n-(1-\epsilon/4)^{2n}\sum_{|A_R|>
      \delta n} \prod_{i=1}^n \frac{\rho(1-q)^{r_i}}{2(e^M-1)r_i}
    \\&\ge (1+\epsilon/4)^n-(1-\epsilon/4)^{2n}\sum_{|A_R|>
      \delta n} \prod_{i=1}^n \frac{\rho(1-q)^{r_i}}{2(e^M-1)r_i}.
  \end{align*}
  To get a contradiction, it is enough to show that the last sum is
  bounded uniformly in $n$. Such a bound is easily obtained from basic
  large deviations for i.i.d.\ Bernoulli random
  variables. Notice that by (\ref{rho}) there is a constant $C$,
  $1-e^{-M}\le C\le 1$, such that $C\rho(1-q)^r/(2r(e^M-1))$,
  $r\in\IN$, is a probability distribution on $\IN$. Consider a
  sequence of i.i.d.\ random variables $(Y_i)_{i\in\IN}$ with this
  distribution. Then $P(Y_i\le R)\le C\rho R/(2(e^M-1))\to 0$ as
  $\rho\to 0$. Thus, for an arbitrary $c>0$ we can choose $\rho$ small
  enough so that for all sufficiently large
  $n$ \[P\left(\sum_{i=1}^n\I_{\{Y_i\le R\}}>\frac{\delta n}{2}\right) \le
  e^{-cn}.\] Since $\{|A_R|>\delta
  n\}\subset\{\sum_{i=1}^n\I_{\{r_i\le R\}}>\delta n/2\}$,
 \[\sum_{|A_R|>
   \delta n} \prod_{i=1}^n \frac{\rho(1-q)^{r_i}}{2(e^M-1)r_i}\le
C^{-n}P\left(\sum_{i=1}^n\I_{\{Y_i\le R\}}>\frac{\delta n}{2}
     \right)\le (Ce^c)^{-n}\le
 1 \] for $c>-\log(1-e^{-M})$ and all large $n$, and we are done.
\end{proof}

We shall need the following comparison lemma. The notation
$\le_{\mathrm{st}}$ (resp.\ $\ge_{\mathrm{st}}$) means
``stochastically smaller'' (resp.\ ``stochastically larger''), where
we use the usual stochastic order (see, for example,
\cite[Sec.\,6.B]{SS07}).
\begin{lemma} \label{lem54} Let $(R_1,R_2,\dots,R_n)$ be distributed
  according to the probability measure
  $\tilde{\Pi}_n^\rho(\cdot):=\Pi_n^\rho(\cdot)/\Pi_n^\rho(\IN^n)$.
  \begin{itemize}
  \item [(a)] For $\Gamma=\Gamma(M)=(1-e^{-M})^{-2}$ 
    \[(R_1,R_2,\dots,R_n)\le_{\mathrm{st}} (Y_1,Y_2,\dots,Y_n),\] where
    $(Y_i)_{1\le i\le n}$ are i.i.d., $Y_i\in\IN$, and for all
    $x\in\IN$
    \begin{equation}
      \label{y}
      P(Y_1\ge x)=1\wedge\left(
      \Gamma\sum_{r=x}^\infty\frac{\rho\,(1-q)^r}{r}\right).
    \end{equation}
  \item [(b)] For $\gamma=1/\Gamma(M)=(1-e^{-M})^2$ 
    \[(R_1,R_2,\dots,R_n) \ge_{\mathrm{st}} (Z_1,Z_2,\dots,Z_n),\] where
    $(Z_i)_{1\le i\le n}$ are i.i.d., $Z_i\in\IN\cup\{0\}$, and for
    all $x\in\IN$
    \begin{equation}
      \label{z}
      P(Z_1\ge
    x)=\gamma\sum_{r=x}^\infty\frac{\rho\,(1-q)^r}{r}.
    \end{equation}
  \end{itemize}
\end{lemma}
\begin{proof}
  Since proofs of both parts are very similar, we prove only part
  (a). By \cite[Th.\ 6.B.B]{SS07}, it is enough to check that there is
  $\Gamma$ such that for all $m\ge 1$ and
  $r_1,\dots,r_{m-1}\in\IN$ \[[R_m\,|\,R_1=r_1,\dots
  R_{m-1}=r_{m-1}]\le_{\mathrm{st}} Y_1, \] where for $m=1$ we agree
  to drop the conditioning. By the definition of $\tilde{\Pi}_n^\rho$,
  (\ref{co2}), and (\ref{defF}) we have for each $x\in\IN$
  \begin{align*}
    &\tilde{\Pi}_n^\rho(R_m\ge x\,|\,R_1=r_1,\dots
    R_{m-1}=r_{m-1})\\&=\frac{\sum\limits_{r_m\ge
        x}\sum\limits_{r_{m+1},\dots,r_n\in\IN} \prod\limits_{i=m}^n
      \rho(1-q)^{r_i}F_M(r_{i-1},r_i,u_{i-1})}{\sum\limits_{r_m,\dots,r_n\in\IN}
      \prod\limits_{i=m}^n
      \rho(1-q)^{r_i}F_M(r_{i-1},r_i,u_{i-1})}\\&\le 1\wedge \left(
      (1-e^{-M})^{-2}\sum\limits_{r_m\ge
        x}\rho(1-q)^{r_m}F_M(r_{i-1},r_i,u_{i-1})\right)\\&\le
    1\wedge \left( (1-e^{-M})^{-2}\sum\limits_{r_m\ge
        x}\frac{\rho(1-q)^{r_m}}{2r_m(e^M-1)}\right)\le 1\wedge
      \left(\Gamma(M)\sum\limits_{r\ge x}\frac{\rho(1-q)^{r}}{r}\right).\qedhere
  \end{align*}
\end{proof}
\begin{remark}
  {\em Observe that due to Proposition~\ref{propcalc} $EZ_1\to \infty$ when
  $\rho\to 0$ or $M\to\infty$.}
\end{remark}
The next corollary follows from the last
lemma by Cram\'er's theorem.
\begin{corollary} \label{corexp} For $ \theta_0 \in(0,1]$, for each
  $H> EY_1$ (see (\ref{y})), and each $h < EZ_1$ (see (\ref{z})) there
  exist strictly positive $c(H)$ and $c'(h)$ such that for all $n$
  large and $\theta\in[\theta_0,1]$
\[\tilde{\Pi}_n^\rho\left(\sum_{1\le i\le \theta
    n}R_i\ge H\theta n\right)\le e^{-c(H) \theta n}\text{ and
}\ \ \tilde{\Pi}_n^\rho\left(\sum_{1\le i\le \theta n}R_i\le h\theta
  n\right)\le e^{-c'(h) \theta n}.\]
\end{corollary} 

Finally, we are ready to convert information about configurations
under $\tilde{\Pi}_n^\rho$ to information under
$Q_{0,y}$. For $1\le n\le k\le y$ let
\begin{align}
  A_{n,k}:&=\left\{(r_1,r_2,\dots,r_k)\in\IN^k\,:\,\sum_{i=1}^{n-1} r_i<y,\
\sum_{i=1}^n r_i\ge y\right\};\nonumber\\
  A_n:&=\left\{(r_1,r_2,\dots,r_n)\in\IN^n\,:\,\sum_{i=1}^n
    r_i=y\right\}\subset A_{n,n}.\label{an}
\end{align}
\begin{lemma} \label{corbond} For each $H>EY_1$  and $h< EZ_1$
  there exist  strictly positive $c_1(H)$ and $c_1'(h)$ such
  that for all $y$ large
  \begin{align*}
    (1-q)^y\sum_{A_n}\prod_{i=1}^n \rho
F_M(r_{i-1}, r_{i}, u_{i-1}) &\leq e^{-c_1(H)
  y}\quad\text{for all $n\in [1,y/H]$};\\
(1-q)^y\sum_{A_n}\prod_{i=1}^n \rho
F_M(r_{i-1}, r_{i}, u_{i-1}) &\leq e^{-c_1'(h) y}\quad\text{for all
  $n\in[y/h,y]$}.
  \end{align*}
 
\end{lemma}
\begin{proof}
  We shall start with the second inequality. If $h< 1$ then there is
  nothing to prove as $[y/h,y]$ is empty. Assume that $1\le h<EZ_1$
  and apply Corollary~\ref{corexp} to $\tilde{\Pi}_y^\rho$ with $n$
  in place of $\theta n$. As
  $A_{n,y}\subset\{(r_1,\dots,r_y)\in\IN^y\,:\, \sum_{i=1}^{n-1}r_i\le
  hn\}$, we get that for $n\in[y/h,y]$
  \begin{equation}
    \label{nn}
    e^{-c'(h)(y/h)}\ge e^{-c'(h)n}\ge \sum_{A_{n,y}} \prod_{i=1}^y
\rho(1-q)^{r_i}F_M(r_{i-1}, r_i, u_{i-1}).
  \end{equation}
If $y\ge n+1$ then we write the right-hand side of the above
inequality as
\begin{multline*}
  \sum_{A_{n,n}} \prod_{i=1}^n
\rho(1-q)^{r_i}F_M(r_{i-1}, r_i,
u_{i-1})\times\left(\sum_{r_{n+1},\dots,r_y\in\IN} \prod_{i=n+1}^y
\rho(1-q)^{r_i}F_M(r_{i-1}, r_i, u_{i-1})\right),
\end{multline*}
and apply (\ref{co2}) to the last summation. Thus, for all
$n\in[y/h,y]$ we have
\begin{align*}
  e^{-c'(h)y/h}&\ge (1-e^{-M})\sum_{A_{n,n}} \prod_{i=1}^n
  \rho(1-q)^{r_i}F_M(r_{i-1}, r_i,
  u_{i-1})\\&\overset{(\ref{an})}{\ge} (1-e^{-M})\sum_{A_n}
  \prod_{i=1}^n \rho(1-q)^{r_i}F_M(r_{i-1}, r_i,
  u_{i-1})\\&=(1-e^{-M})(1-q)^y\sum_{A_n} \prod_{i=1}^n \rho
  F_M(r_{i-1}, r_i, u_{i-1}).
\end{align*}

The proof of the first inequality is similar. Notice that for all
$n\in[1,y/H]$ we have $A_{n,y}\subset\{(r_1,\dots,r_y)\in\IN^y\,:\,
\sum_{i=1}^{[y/H]}r_i\ge H(y/H)\}$. Again by Corollary~\ref{corexp}
for $\tilde{\Pi}_y^\rho$ and $y/H$ in place of $\theta n$ we get
\begin{equation*}
    e^{-c(H)y/H}\ge \sum_{A_{n,y}} \prod_{i=1}^y
\rho(1-q)^{r_i}F_M(r_{i-1}, r_i, u_{i-1}),\quad  \text{for all $n\in[1,y/H]$}.
  \end{equation*}
The rest of the proof follows the proof of the second inequality.
\end{proof}

\begin{lemma} \label{lemfinal} Let $h$ be as in Lemma~\ref{corbond}.
  There exists a strictly positive $c_2=c_2(\rho,M)$ such that for $y$ large
  \[(1-q)^y\sum_{1\le n< y/h}\sum_{A_n}
  \prod_{i=1}^n \rho F_M(r_{i-1}, r_i, u_{i-1}) \geq c_2.\]
\end{lemma}
\begin{proof}
  We start with (\ref{nn}) and sum up over $n\in[y/h,y]$. The number
  of terms in this summation does not exceed $y$, therefore, for all
  sufficiently large $y$ the sum is less than $(1-e^{-M})/2$. Since
  the sum over $n\in[0,y)$ is equal to $\Pi^\rho_y(\IN^y)\ge
  (1-e^{-M})$, we conclude that \[\frac12\,(1-e^{-M})\le
  \sum_{1\le n< y/h}\sum_{A_{n,y}} \prod_{i=1}^y \rho(1-q)^{r_i}F_M(r_{i-1},
  r_i, u_{i-1}).\] Just as in the previous proof, if $y\ge n+1$ then
  we perform first the summation over $r_{n+1},\dots,r_y\in\IN$ and
  apply (\ref{co2}) to get \[\frac12\,(1-e^{-M})\le
  (1-e^{-M})^{-1}\sum_{1\le n< y/h} \sum_{A_{n,n}} \prod_{i=1}^n
  \rho(1-q)^{r_i}F_M(r_{i-1}, r_i, u_{i-1}). \] Next, we replace
  $F_M(r_{n-1},r_n,u_{n-1})$ with its upper bound $F_M(r_{n-1},y-\sum_{i=1}^{n-1}
  r_i,u_{n-1})$ and sum over $r_n$ from $y-\sum_{i=1}^{n-1} r_i$ to
  infinity. We obtain
  \begin{multline*}
    \frac12\,(1-e^{-M})\le \frac{1}{q(1-e^{-M})}\,\sum_{1\le n< y/h}
    \sum_{A_n} \prod_{i=1}^n \rho(1-q)^{r_i}F_M(r_{i-1}, r_i,
    u_{i-1})\\=\frac{(1-q)^y}{q(1-e^{-M})}\,\sum_{1\le n< y/h}
    \sum_{A_n} \prod_{i=1}^n \rho F_M(r_{i-1}, r_i, u_{i-1}).
  \end{multline*}
This gives the desired statement with $c_2=q(1-e^{-M})^2/2$.
\end{proof}
\begin{corollary}\label{nc}
  Let $\rho=p/(1-p)$, $q=q(\rho)$ be as defined in
  Corollary~\ref{corbase}, and $c_2$ be the same as in
  Lemma~\ref{lemfinal}. Then in (\ref{23}), \[Z_{0,y}^{-1}(1-p)^y\left(\frac{e^M}{\rho}+1
  \right)\le c_2^{-1}(1-q)^y.\] 
\end{corollary}

\subsection{Final step: construction of 
  $\Omega_y^1$ and $\Omega_y^2$} \label{last}

Throughout this subsection we suppose that, for a given environment
$\omega$, the occupied sites in $(0,y)$ are $\{x_1,x_2,\dots,x_{n-1}\}$, where
$0=x_0<x_1<\dots<x_{n-1}<y=:x_n$. As before, we set $r_i=x_i-x_{i-1}$,
$i=1,2,\dots,n$. 

Let $\rho=p/(1-p)$. Fix an $\epsilon>0$
and set $h=e^{2(1- \epsilon/2) (e^M - 1)/\rho}$. We claim that $h<EZ_1$
for all sufficiently small $\rho$ ($M$ is fixed). Indeed, by (\ref{z})
and
Proposition~\ref{propcalc} \[EZ_1=\frac{\gamma(M)\rho(1-q(\rho))}{q(\rho)},\
\text{where} \ \frac{1}{q(\rho)}=e^{2(e^M-1)(1+o(\rho))/\rho}\
\text{as }\rho\to 0.\] 

Let ${\cal A}_n:=\{\omega\in\Omega\,:\, N(\omega,y)=n,\
(R_1(\omega),R_2(\omega)\dots,R_n(\omega))\in A_n\}$, where $A_n$ was
defined in (\ref{an}). Then by Lemma~\ref{corbond} and
Corollary~\ref{nc} we have that for all sufficiently small $p$ and
large $y$ ($c_1$ and $c_2$ do not depend on
$y$)
\begin{equation}
  \label{om}
  Q_{0,y}\left(\cup_{\{n:\ y/h\le n\le y\}}{\cal A}_{n}\right)\le
\frac{y}{c_2}\,e^{-c_1'(h)y}.
\end{equation}

Now we are ready to construct $\Omega^1_y$. Let
$\Omega_y^1=\cup_{\{n:\,1\le n<
  y/h\}}{\cal A}_n$. The bound (\ref{om}) implies that $Q_{0,y}(\Omega_y^1)\ge
1/2$ for all sufficiently large $y$. Thus, (L1) is satisfied. Let
$\omega\in\Omega_y^1$. By Lemma~\ref{U} and Cauchy-Schwarz
inequality, \[E_{Q^\omega_{0,y}}\tau_y\ge \frac13 \sum_{i=1}^nr_i^2\ge
\frac{y^2}{3n}\ge \frac{y^2}{3(y/h)}\ge
\frac{hy}{3}\ge\frac{y}{3}\,e^{(1-\epsilon)K(p,M)}.\] This gives
us (L2) and the desired lower bound (\ref{lb}).

The upper bound is somewhat more involved. Lemma~\ref{U} provides us
with the following estimate:
\[E_{Q^\omega_{0,y}}\tau_y\le
\frac{1}{3(1-e^{-M})}\sum_{j=1}^nr_j^2.\] We set
$\Omega_y^2=\cup_{\{n:\,1\le n< y/h\}}{\cal B}_n$, where \[{\cal
  B}_n=\left\{\omega\in\Omega\,:\, N(\omega,y)=n,\
  \sum_{i=1}^nR_i(\omega)=y,\ \sum_{i=1}^nR_i^2(\omega)\le
  \frac{C_3 n}{\log^2(1/(1-q))}\right\},\] and $C_3$ is a sufficiently
large constant, which we shall determine later (see
Lemma~\ref{corfin}).  The term $\log(1/(1-q))$ gives us the right
scaling, since under the probability measure $\tilde{\Pi}^\rho_n$ the
expected ``gap'' size is roughly of order $1/q\sim 1/\log(1/(1-q))$ as
$q\to 0$. We shall show that with this scaling the constant $C_3$ can
indeed be chosen uniformly over all small $p$ and large $M$.

At first, we check that $\Omega_y^2$ satisfies (U2). Let
$\omega\in\Omega_y^2$. Then $n< y/h$
and \[E_{Q^\omega_{0,y}}\tau_y\le \frac{C_3 n}{3(1-e^{-M})\log^2(1/(1-q))}\le
\frac{C_3 y}{3(1-e^{-M})q^2h}. \] By Proposition~\ref{propcalc}, $q^{-1}\le
\exp(2(1+\epsilon/4)(e^M-1)/\rho)$ for all sufficiently small
$\rho$. Substituting the expressions for $h$ and $\rho$ we get the
desired upper bound \[E_{Q^\omega_{0,y}}\tau_y\le
\frac{C_3y}{3(1-e^{-M})}\,e^{(1+\epsilon)K(p,M)}.\]

Our last task is to establish (U1). Notice that \[\Omega\setminus
\Omega_y^2\subset \left(\cup_{\{n:\,y/h\le n\le y\}} {\cal
    A}_n\right)\cup\left(\cup_{\{n:\,1\le n\le y/H\}} {\cal
    A}_n\right)\cup(\cup_{y/H\le n\le y/h}{\cal D}_n),\] where $h$ and $H$
are chosen as in Lemma~\ref{corbond}, and \[{\cal D}_n=
\left\{\omega\in\Omega\,:\,N(\omega,y)=n,\
  \sum_{i=1}^n R_i(\omega)=y,\ 
  \sum_{i=1}^nR_i^2(\omega)> \frac{C_3 n}{\log^2(1/(1-q))}\right\}.\]
Therefore, by (\ref{om}), Lemma~\ref{corbond}, and Corollary~\ref{nc},
\begin{equation}
  \label{qu}
  Q_{0,y}(\Omega\setminus \Omega_y^2)\le
\frac{y}{c_2}\,e^{-c_1'(h)y}+\frac{y}{c_2}\,e^{-c_1(H)y}+y\max_{y/H\le
  n\le
  y/h}Q_{0,y}({\cal D}_n).
\end{equation}
Again we estimate first probabilities
$\tilde{\Pi}_n^\rho$ of the relevant events.
\begin{lemma}
 \label{corfin}
 For $ \theta_0 \in(0,1]$ there are $c_3,C_3>0$ such that for all
 sufficiently large $n$ and
 $\theta\in[\theta_0,1]$ \[\tilde{\Pi}_n^\rho\left( \sum_{i=1}^{\theta
     n}R_i^2>\frac{C_3 \theta n}{\log^2(1/(1-q))}\right)\le
 e^{-c_3\sqrt{\theta n}}.\]
\end{lemma}
\begin{proof}
  By part (a) of Lemma~\ref{lem54} \[\tilde{\Pi}_n^\rho\left(
    \sum_{i=1}^{\theta n}R_i^2>\frac{C_3 \theta
      n}{\log^2(1/(1-q))}\right)\le P\left(\sum_{i=1}^{\theta
      n}Y_i^2>\frac{C_3 \theta n}{\log^2(1/(1-q))}\right).\] Thus we
  need a large deviations upper bound for a sequence of i.i.d.\ random
  variables $Y_i^2\log^2(1/(1-q))$, $i\in\IN$, which have
  sub-exponential tails: there are $\ell\in\IN$, $\rho_0,M_0>0$
  such that for all $x\ge \ell$ and
  $(\rho,M)\in[0,\rho_0]\times[M_0,\infty)$
  \begin{multline*}
    P(Y_i^2\log^2(1/(1-q))\ge x)=\Gamma(M)\sum_{r\ge
        \sqrt{x}/\log(1/(1-q))}\frac{\rho (1-q)^r}{r}\\ \le
    \frac{\Gamma(M)\rho(1-q)^{-\sqrt{x}/(\log(1-q)}}
      {\sqrt{x}/\log(1/(1-q))}\sum_{r\ge 0}(1-q)^r\le 
    \frac{\Gamma(M)\rho }{\sqrt{\ell}}\,e^{-\sqrt{x}}q^{-1}\log\frac1{1-q}\le 
    \,e^{-\sqrt{x}}.
  \end{multline*}
 The statement of the lemma now follows from Lemma~\ref{expbound}.
\end{proof}
It only remains to convert (as we have done before) the previous
result into a bound on $Q_{0,y}$ probability.
\begin{lemma}
  \label{corbond'} Let $H$ be chosen as in Lemma~\ref{corbond}. There
  is strictly positive $c_4=c_4(H)$ such that for all $y/H\le n\le y$
\[
 (1-q)^y\sum_{D_n}\prod_{i=1}^n \rho F_M(r_{i-1}, r_{i}, u_{i-1})
 \leq e^{-c_4 \sqrt{y}}.
\] 
\end{lemma}
\begin{proof}
  For $1\le n\le k\le y$ let
\begin{align*}
  D_{n,k}&:=\left\{(r_1\dots,r_k)\in\IN^k\,:\,\sum_{i=1}^n r_i<y,\
    \sum_{i=1}^n r_i\ge y,\sum_{i=1}^nr_i^2> 
    \frac{C_3 n}{\log^2(1/(1-q))}\right\}\\ D_n&:=
\left\{(r_1,r_2,\dots,r_n)\in\IN^n\,:\,\sum_{i=1}^nr_i=y,\
  \sum_{i=1}^nr_i^2> \frac{C_3 n}{\log^2(1/(1-q))}\right\}.
\end{align*}
Since $D_{n,y}\subset \{(r_1,\dots,r_y)\in\IN^y\,:\,
\sum_{i=1}^nr_i^2> C_3 n\log^{-2}(1/(1-q))\}$ and $n\ge y/H$,
we have by Lemma~\ref{corfin} that \[e^{-c_3\sqrt{y/H}}\ge
  e^{-c_3\sqrt{n}}\ge \sum_{D_{n,y}} \prod_{i=1}^y
\rho(1-q)^{r_i}F_M(r_{i-1}, r_i, u_{i-1}).\] The rest of the proof
follows the one of Lemma~\ref{corbond} with $D_{n,y}$ and $D_n$ in
place of $A_{n,y}$ and $A_n$.
\end{proof}
Lemma~\ref{corbond'} and Corollary \ref{nc} imply that \[\max_{y/H\le
  n\le y/h}Q_{0,y}({\cal D}_n)\le c_2^{-1}e^{-c_4 \sqrt{y}}.\] 
Together with (\ref{qu}) this estimate establishes (U1).
 
\appendix
 
\section{Proofs of technical lemmas}

\begin{proof}
  [Proof of Proposition~\ref{basic1}] Let $u_0=u_n=0$, and
  $u_k=E^k(\tau_n;\tau_n<\tau_0)$. Then 
  \begin{align*}
    u_k&=E^k(\tau_n;\tau_n<\tau_0)\\&=
    \frac12\,E^{k+1}(\tau_n+1;\tau_n<\tau_0)+
    \frac12\,E^{k-1}(\tau_n+1;\tau_n<\tau_0)\\&=
    \frac12\,u_{k+1}+\frac12\,u_{k-1}+\frac12\,P^{k+1}(\tau_n<\tau_0)+\frac12\,
    P^{k-1}(\tau_n<\tau_0)\\&=\frac12\,u_{k+1}+\frac12\,u_{k-1}+
    \frac12\,\frac{k+1}{n}+\frac12\,\frac{k-1}{n}=
    =\frac12\,u_{k+1}+\frac12\,u_{k-1}+\frac{k}{n}.
  \end{align*}
Denoting by $\Delta u_k$ the discrete Laplacian of $u_k$, we get that
  \[-\frac12\,\Delta u_k=\frac{k}{n},\ k=0,1,2,\dots,n.\]
  The continuous analog is $-\Delta u=2x$, $u(0)=u(1)=0$, $x\in[0,1]$,
  and it is easy to solve: $u(x)=x(1-x)(1+x)/3$. From the scaling
  property of the random walk  we conclude
  that $u_k=(3n)^{-1}k(n-k)(n+k)$. It is also easy to check
  directly that
  this expression indeed solves the equation $-\frac12\,\Delta
  u_k=k/n,\ k=0,1,2,\dots,n$, and satisfies the boundary conditions.
\end{proof}
\begin{proof}
[Proof of Proposition~\ref{basic2}]Notice that we can omit
  $\I_{\{\tau_y<\infty\}}$ from (\ref{qpm}), (\ref{qnf}), and
  (\ref{que}), since on the event $\{\tau_y=\infty\}$ the exponential
  function in the integrand vanishes $P$-a.s..  By the strong Markov
  property, for every $\lambda\ge 0$ and $y\in\IN$ we can write 
\begin{align*}
  \frac1y\log
  E\left(e^{-\sum_{n=0}^{\tau_y-1}(\lambda+V(S_n,\omega))}\right)&=
  \frac1y\sum_{k=0}^{y-1}\log
  E^k\left(e^{-\sum_{n=0}^{\tau_{k+1}-1}(\lambda+V(S_n,\omega))}\right),\\
  \frac{E_{Q^\omega_y}\tau_y}{y}&=\frac1y\sum_{k=0}^{y-1}
  E_{Q^{\omega,k}_{k+1}}(\tau_{k+1}).
\end{align*}
The ergodic theorem implies that
\begin{align*}
  \alpha_{\lambda+V}(1)& =-\IE\left(\log
E\left(e^{-\sum_{n=0}^{\tau_1-1}(\lambda+V(S_n,\omega))}\right)\right),\\
\frac{1}{v^\qu}&=\IE\left(E_{Q^\omega_1}(\tau_1)\right)=
\IE\left(E^\omega(\tau_1\,|\,\tau_1<\infty)\right).
\end{align*}
Next we show that
$\IE\left(E^\omega(\tau_1\,|\,\tau_1<\infty)\right)<\infty$. Let
$\Lambda(t)=-\log \IE(e^{-tV(0,\cdot)})$ and
$\ell(x)=\sum_{n=0}^{\tau_1-1}\I_{\{S_n=x\}}$. Then, since
$P^\omega(\tau_1<\infty)\ge e^{-V(0,\omega)}/2$,
\begin{align*}
  \IE\left(E^\omega(\tau_1\,|\,\tau_1<\infty)\right)&\le
  \IE\left(2e^{V(0,\cdot)}E\left(\tau_1
      e^{-\sum_{n=0}^{\tau_1-1}V(S_n,\cdot)}\right)\right) \le
  2E\left(\tau_1e^{-\sum_{x<0}\Lambda(\ell(x))}\right) \\&\le
  2P(\tau_1<\tau_{-1})+2\sum_{k=1}^\infty
  E\left(\tau_1e^{-k\Lambda(1)};\tau_{-k}<\tau_1<\tau_{-(k+1)}\right)
  \\ &\le 1+2\sum_{k=1}^\infty e^{-k\Lambda(1)}
  E\left(\tau_1;\tau_1<\tau_{-(k+1)}\right)
\overset{\mathrm{Prop.\,2.1}}{<}\infty.
\end{align*}
The function $\alpha_{\lambda+V}(1)$ is concave and
non-decreasing (see \cite[p.\,272]{Ze98}). Taking the right derivative
of $\alpha_{\lambda+V}(1)$ at $\lambda=0$ we obtain the last statement
of proposition. 
\end{proof}
 
\begin{proof}[Proof of Lemma~\ref{elem1}]
The proof is very simple. If $a_1=0$ then a trivial bound is given
by the survival probability $e^{-M}$. For $a_1\in\IN$ we have
\begin{align*}
  P&^{\omega}(\tau_1>\tau_{-a_1}\,|\,\tau_1<\infty)\\&=\frac{P^{\omega}
    (\tau_1<\infty\,|\,\tau_1>\tau_{-a_1})P^{\omega}(\tau_1>\tau_{-a_1})}
  {P^{\omega}(\tau_1<\infty)}=\frac{1}{a_1+1}\frac{P^{\omega,-a_1}
    (\tau_1<\infty)}{P^{\omega}(\tau_1<\infty)}.
\end{align*}
It is obvious that the last ratio is bounded by $e^{-M}$ but we shall need
an improvement of this estimate for the proof of Lemma~\ref{elem2}. We
have
\begin{multline*}
  \frac{P^{\omega,-a_1}
    (\tau_1<\infty)}{P^{\omega}(\tau_1<\infty)}=
  P^{\omega,-a_1}(\tau_0<\infty)=e^{-M}\left(\frac{1}{2a_1}+\frac{a_1-1}{2a_1}
    P^{\omega,-a_1}(\tau_0<\infty)\right.\\\left.+\frac{1}{2}
    P^{\omega,-a_1-1}(\tau_0<\infty\,|\,\tau_{-a_1}<\infty)
    P^{\omega,-a_1-1}(\tau_{-a_1}<\infty)\right)\\\le
  e^{-M}\left(\frac{1}{2a_1}+\left(1-\frac{1}{2a_1}\right)P^{\omega,-a_1}
    (\tau_0<\infty)\right).
\end{multline*}
This gives
\begin{equation}
  \label{a1}
  \frac{P^{\omega,-a_1}
    (\tau_1<\infty)}{P^{\omega}(\tau_1<\infty)}\le
\frac{e^{-M}}{2a_1(1-e^{-M})+e^{-M}},
\end{equation}
and the statement of the lemma follows.
\end{proof}
\begin{proof}[Proof of Lemma~\ref{elem2}]
  Suppose $a_1=0$. Recall that $\tau_0$ is the time of the first
  return to $0$. We just have to bound the expected number of visits
  to $0$ before $\tau_1$:
  \begin{align*}
    E^\omega(\I_{\{\tau_0<\tau_1\}}&\sum_{n=0}^{\tau_1-1}
    \I_{\{S_n=0\}}\,|\,\tau_1<\infty) \\&\le
    P^{\omega}(\tau_0<\tau_1\,|\,\tau_1<\infty)+
    E^\omega(\I_{\{\tau_0<\tau_1\}}\sum_{n=1}^{\tau_1-1}
    \I_{\{S_n=0\}}\,|\,\tau_1<\infty)\\ &\le e^{-M}+
    \frac{P^{\omega}(\tau_0<\tau_1)}{P^{\omega}(\tau_1<\infty)}E^\omega
    \left(\sum_{n=0}^{\tau_1-1}\I_{\{S_n=0\}}\I_{\{\tau_1<\infty\}} \right)
    \\&\le e^{-M}+ E^\omega
    \left(\sum_{n=0}^{\tau_1-1}\I_{\{S_n=0\}}\I_{\{\tau_1<\infty\}}
    \right) \le 3e^{-M}.
  \end{align*}
  We used the following obvious facts: $P^{\omega}(\tau_0<\tau_1)\le
  (2e^M)^{-1}$ and $P^{\omega}(\tau_1<\infty)\ge
  (2e^M)^{-1}$. The last inequality in the multi-line formula above
  is obtained by considering the Markov chain which starts at $0$,
  gets killed with probability $(1-e^{-M})$ after each visit to $0$,
  otherwise goes with equal probabilities to $1$ and $0$, and always
  gets absorbed at $1$.  For such chain the expected time to hit $1$
  restricted to the event that it reaches $1$ is equal to
  $2e^{-M}/(2-e^{-M})^2\le 2e^{-M}$.

 Assume now that $a_1\in\IN$.   
  Since $P^{\omega}(\tau_1<\infty)\ge 1/2$, we can replace the
  conditioning on the event $\{\tau_1<\infty\}$ with the intersection
  at the cost of factor 2. The time spent in $I$ before hitting $1$ is
  the sum of three terms: (1) $\tau_{-a_1}$; (2) the total time spent
  in excursions to $I$ from $-a_1$ that end up in $-a_1$; (3) the time
  needed to get from $-a_1$ to $1$ without returning to $-a_1$.

  Term (1) is estimated using Lemma~\ref{basic2}. We get $Ca_1$ times
  $P^{\omega,-a_1}(\tau_1<\infty)$. By (\ref{a1}) we have the required
  bound $Ce^{-M}/(1-e^{-M})$. Term (3) is bounded by the product of
  $P^{\omega}(\tau_{-a_1}<\tau_1)$, $e^{-M}$, and
  $1+E^{-a_1+1}(\tau_1;\tau_1<\tau_{-a_1})$. The latter is again
  bounded by $Ca_1$. Thus, (3) does not exceed $Ce^{-M}$. Term (2) is
  the sum of a random number of durations of excursions. The expected
  duration of one such excursion is bounded by $Ca_1$
  (Lemma~\ref{basic2}), and the number of them is at most geometric
  with expectation $e^{-M}/(1-e^{-M})$. The total is multiplied by
  $P^{\omega}(\tau_{-a_1}<\tau_1)$, since to have such excursion the
  path has to get to $-a_1$ before hitting $1$. The strong Markov
  property implies the desired bound. Adding the three terms we get
  the statement of the lemma.
\end{proof}
\begin{proof}[Proof of Lemma~\ref{elem3}]
  This is a rough bound, which is sufficient for our purposes. Just as
  in the proof of Lemma~\ref{elem2}, we replace the conditioning on
  the event $\{\tau_1<\infty\}$ by the intersection with
  $\{\tau_1<\infty\}$ at a cost of factor $2e^M$ (here we can not
  exclude the case $a_1=0$). The time spent in $I_j$, $j\in\IN$, is
  equal to the sum of durations of excursions contained in $I_j$ from
  $-a_{j+1}$ and $-a_j$ as well as crossings between the end points of
  $I_j$. The total number of such excursions and crossings is again at
  most geometric with expectation $e^{-M}/(1-e^{-M})$ and the expected
  duration is bounded by $C(a_{j+1}-a_j)$ due to
  Lemma~\ref{basic2}. But to have a chance to undergo at least one
  such excursion or crossing the walk has to reach $-a_j$ prior to
  hitting $1$, survive at least one visit to each $-a_i$, $1\le i<j$,
  and after completing the excursions reach $1$ before returning to
  $-a_j$ surviving the final run through $\{-a_i,j\ge i\ge 1\}$. Thus,
  the expectation, which we want to estimate is bounded
  by \[\frac{C(a_{j+1}-a_j)}{1-e^{-M}}\,\frac{e^{-M(2j-1)}}{(a_j+1)^2}.\]
  Replacing $2j-1$ with $j$ and $a_j+1$ with $a_1+1$ we obtain the
  desired upper bound.
\end{proof}

\begin{proof}[Proof of Lemma~\ref{F}] By simple
  gambler's ruin considerations we have
\[u_1=\frac{e^{-V(0,\omega)}}{2r_1}=e^{M-V(0,\omega)}F_M(0,r_1,0),\] and for $n>1$ 
  \begin{equation*}
    u_n = e^{-M}\left(\frac{1}{2r_n} + \frac{u_{n-1}u_n
      }{2r_{n-1}}+ u_n{\left(1-
        \frac{1}{2r_n} - \frac{1}{2r_{n-1}}
      \right)}\right).
  \end{equation*}
  Solving the last equation for $u_n$ we get for $n>1$
  \begin{equation*}
    u_n = \frac{e^{-
      M}}{2r_n}\left(1- e^{- M}\left(1- \frac{1}{2r_n} -
      \frac{1-u_{n-1}}{2r_{n-1}}\right)\right)^{-1}=F_M(r_{n-1},r_n,u_{n-1})
  \end{equation*}
as claimed.
\end{proof}

\begin{lemma}
  \label{A}
For every $p\in[0,1)$, $M\in[0,\infty)$,  $y\in\IN$, and 
$\omega\in\Omega$\[E_{Q^\omega_{0,y}}\tau_y\le 1+2y^2.\]
\end{lemma}
\begin{proof}
Let $\ell_y(x)=\sum_{n=0}^{\tau_y-1}\I_{\{S_n=x\}}$.
We need to estimate
\begin{equation}
  \label{a}
  E_{Q^\omega_{0,y}}\tau_y=1+\sum_{x=1}^{y-1}
  E_{Q^\omega_{0,y}}\ell_y(x)=1+\sum_{x=1}^{y-1} 
\sum_{m=0}^\infty Q^\omega_{0,y}(\ell_y(x)>m).
\end{equation}
Denote by ${\cal G}_n$ the sigma-algebra generated by the simple
random walk up to time $n$. Then by the strong Markov property of the
simple random walk we have 
\begin{align*}
  &Z^\omega_{0,y}Q^\omega_{0,y}(\ell_y(x)>m)=E\left(e^{-\sum_{n=0}^{\tau_y-1}
      V(S_n,\omega)}\I_{\{\tau_y<\tau_0,\,\tau_x^{(m+1)}<\tau_y\}}\right)
  \\&\le
  E\left(e^{-\sum_{n=0}^{\tau_x-1}V(S_n,\omega)}\I_{\{\tau_x<\tau_0\}}
    E\left(e^{-\sum_{n=\tau_x^{(m+1)}} ^{\tau_y-1}V(S_n,\omega)}
      \I_{\{\tau_y<\tau_0,\tau_x^{(m+1)}<\tau_y\}}\,\big|\, {\cal
        G}_{\tau_x}\right)\right)
  \\&=E\left(e^{-\sum_{n=0}^{\tau_x-1}V(S_n,\omega)}\I_{\{\tau_x<\tau_0\}}
    E^x\left(e^{-\sum_{n=\tau_x^{(m)}} ^{\tau_y-1}V(S_n,\omega)}
      \I_{\{\tau_y<\tau_0,\tau_x^{(m)}<\tau_y\}}\right)\right) \\&=
  E\left(e^{-\sum_{n=0}^{\tau_x-1}V(S_n,\omega)}\I_{\{\tau_x<\tau_0\}}E^x
    \left(e^{-\sum_{n=\tau_x^{(m)}} ^{\tau_y-1}V(S_n,\omega)}
      \I_{\{\tau_y<\tau_0\}}\ \big|\,\tau_x^{(m)}<\tau_y\wedge
      \tau_0\right)P^x\left(\tau_x^{(m)}<\tau_y\wedge \tau_0\right)
  \right) \\&=
  E\left(e^{-\sum_{n=0}^{\tau_x-1}V(S_n,\omega)}\I_{\{\tau_x<\tau_0\}}E^x\left(e^{-\sum_{n=0}
        ^{\tau_y-1}V(S_n,\omega)} \I_{\{\tau_y<\tau_0\}}\right)\right)
  P^x\left(\tau_x^{(m)}<\tau_y\wedge \tau_0\right)
  \\&=Z^\omega_{0,y}P^x(\tau_x^{(m)}<\tau_y\wedge \tau_0).
\end{align*}
Substituting this into (\ref{a}) and taking into account
that \[P^x(\tau_x^{(m)}<\tau_y\wedge \tau_0) =
\left(\frac{x-1}{2x}+\frac{y-x-1}{2(y-x)}\right)^m=
\left(1-\frac{y}{2x(y-x)}\right)^m,\] we get
\[E_{Q^\omega_{0,y}}\tau_y\le 1+\sum_{x=1}^{y-1} 
\sum_{m=0}^\infty\left(1-\frac{y}{2x(y-x)}\right)^m=
1+\frac2{y} \sum_{x=1}^{y-1} x(y-x)\le 1+2y^2.\]
\end{proof}
\begin{lemma}
  \label{U}
  Fix an arbitrary $p\in[0,1)$, $M\in[0,\infty)$, $\omega\in\Omega$,
  and $y\ge 1$. Let ${\cal O}:=\{x_1,x_2,\dots,x_{n-1}\}$ be the set of
  all occupied sites in $(0,y)$,
  $0=:x_0<x_1<\dots<x_{n-1}<x_n:=y$. Set $r_j=x_j-x_{j-1}$,
  $j=1,2,\dots,n$. Then \[\frac13\,\sum_{j=1}^nr_j^2\le E_{Q^\omega_{0,y}}\tau_y\le
  \frac{1}{3(1-e^{-M})}\sum_{j=1}^nr_j^2.\]
\end{lemma}
 
\begin{proof}
  If $n=1$, i.e.\ ${\cal O}=\emptyset$, then the statement follows
  from the properties of the standard random walk (see
  Proposition~\ref{basic1}). Assume that $n\ge 2$, and, therefore,
  $y\ge n\ge 2$. The proof is based on the decomposition of random
  walk paths according to the total number of visits, $k$, to ${\cal
    O}$ before hitting $y$ and to the order, in which the walk visits
  points of ${\cal O}$. For each $k\ge n-1$, such orderings are
  represented by admissible sequences $(y_1,y_2,\dots,y_k)\in{\cal
    O}^k$. For example, for $n=4$ and $k=9$ a sequence
  $(x_1,x_2,x_2,x_3,x_2,x_1,x_2,x_3,x_4)$ is admissible. The set of
  all admissible sequences of length $k$ will be denoted by ${\cal
    Y}_k$ and the set of all random walk paths corresponding to a
  given admissible $\mathbf{y}_k\in{\cal Y}_k$ by
  $W_{\mathbf{y}_k}$. Define $\sigma_0=0$, $\sigma_i=
  \inf\{j>\sigma_{i-1}:\, S_j=y_i\}$, $i\in\IN$. With this notation,
  we have $\tau_y=\sum_{i=1}^k(\sigma_i-\sigma_{i-1} )$ and
  \begin{equation}\label{pd}
    E_{Q^\omega_{0,y}}\tau_y=
    \dfrac{\sum_{k=n-1}^\infty e^{-Mk-V(0)}\sum_{\mathbf{y}_k\in{\cal
          Y}_k}E^0\left(\sum_{i=1}^k(\sigma_i-\sigma_{i-1}
        )\I_{W_{\mathbf{y}_k}}\right)}{\sum_{k=n-1}^\infty e^{-Mk-V(0)}
      \sum_{\mathbf{y}_k\in{\cal
          Y}_k}E^0(\I_{W_{\mathbf{y}_k}})}.
  \end{equation}
  It is easy to see from Proposition~\ref{basic1} that  
 \begin{equation}
   \label{ct}
   E^0(\tau_m\,|\,\tau_m<\tau_0^{(2)})=\frac{m^2+2}{3}\ge \frac{m^2}{3} 
   \ \text{and}\   
   E^\ell(\tau_m\,|\,\tau_m<\tau_0^{(2)})=\frac{m^2-\ell^2}{3}\le \frac{m^2}3.
 \end{equation}
 for all $m\in\IN$ and
 $\ell\in\{1,\dots,n-1\}$. Since the random walk has to cross every
 interval $(x_{j-1},x_j)$, $j=1,\dots, n$, at least once, the strong
 Markov property of the standard random walk, (\ref{pd}), and the
 first inequality in (\ref{ct}) immediately give us the claimed lower bound.

 We turn now to the upper bound.  Using the strong
 Markov property and the second inequality in (\ref{ct}), we can
 estimate the right-hand side of (\ref{pd}) from above by
  \begin{multline*}
    \frac{\left(\sum_{j=1}^nr_j^2\right)\sum\limits_{k=n-1}^\infty\left(
        e^{-Mk-V(0)} \sum\limits_{\mathbf{y}_k\in{\cal Y}_k}
        \left(\max\limits_{1\le j\le
            n-1}\sum_{i=1}^k\I_{\{y_i=x_j\}}\right)
        P^0(W_{\mathbf{y}_k})\right)}{3\sum_{k=n-1}^\infty \left(
        e^{-Mk-V(0)} \sum\limits_{\mathbf{y}_k\in{\cal Y}_k}
        P^0(W_{\mathbf{y}_k})\right)}\\\le \frac13
    \left(\sum_{j=1}^nr_j^2\right)\max_{x\in{\cal
        O}}E_{Q^\omega_{0,y}}\ell_y(x),
  \end{multline*}
  where $\ell_y(x)=\sum_{j=0}^{\tau_y-1}\I_{\{S_j=x\}}$. The last
  term can be handled in the same way as in the proof of
  Lemma~\ref{A}. The only difference is that each additional visit to
  an occupied site adds the factor $e^{-M}$, and it is this
  factor that plays a major role. For every $x\in{\cal O}$ we have
  \begin{multline*}
    E_{Q^\omega_{0,y}}\ell_y(x)= \sum_{m=0}^\infty
  Q^\omega_{0,y}(\ell_y(x)>m)
\le\sum_{m=0}^\infty e^{-Mm}
    P^x(\tau^{(m+1)}_x<\tau_y\wedge \tau_0)\le \frac{1}{1-e^{-M}}.
  \end{multline*}
This completes the proof. 
\end{proof}
\begin{lemma}
 \label{expbound}
 Let $X_i$ be i.i.d.\ non-negative random variables such that $P(X_1
 \geq x) \leq e^{- \sqrt x}$ for all $x \geq n_0$. There exist constants
 $C,c\in(0,\infty)$ such that for all $n$ large
 \[
 P \left(\sum_{i=1}^n X_i \geq Cn  \right) \leq e ^{-c \sqrt n}.
 \]
 \end{lemma}
 This fact is contained in \cite[Theorem 1.1,\ p.\,748]{Na79} but for
 convenience of the reader we give a short proof. 
\begin{proof}
  Fix an $\epsilon\in(0,1)$ and let $X^{\epsilon n}_i= X_i \I_{\{X_i\le
      \epsilon n\}}$, $i\in\IN$. Then \[\left\{ \sum_{i=1}^n X_i
    \geq Cn \right\}\subset \Big(\bigcup_{i=1}^n \{X_i \geq \epsilon
  n\}\Big)\cup \left\{ \sum_{i=1}^n X^{\epsilon n}_i \geq Cn
  \right\}.\]
The probability of the first union is bounded by $n e ^{-
    \sqrt{\epsilon n}} $ (for $\epsilon n > n_0$), which is less than $
  e ^{- \sqrt {\epsilon n/2}} $ for all large $n$.  It remains to
  bound the probability of the last event.  By Chebyshev's
  inequality \[P\left( \sum_{i=1}^n X^{\epsilon n}_i \geq Cn \right) \leq
  e^{-\frac{C\sqrt n}{ 2}}\left(E\exp\left(\frac{X^{\epsilon
        n}_1}{2\sqrt{n}}\right)\right)^n.\]
We claim that $E\exp\left(X^{\epsilon
        n}_1/(2\sqrt{n})\right)<1$ for all sufficiently large
    $n$. Indeed,
    \begin{align*}
      E\exp\left(\frac{X^{\epsilon
            n}_1}{2\sqrt{n}}\right)&=\frac{1}{2\sqrt{n}}\int_0^{\epsilon
        n} e^{x/(2\sqrt{n})} P(X_1^{\epsilon n}>x)\,dx\\&\le
      \frac{1}{2\sqrt{n}}\int_0^{\sqrt{n}} e^{x/(2\sqrt{n})}
      \,dx+\frac{1}{2\sqrt{n}}\int_{\sqrt{n}}^{\epsilon n}
      e^{x/(2\sqrt{n})} P(X_1^{\epsilon n}>x)\,dx\\&\le
      \sqrt{e}-1+\frac{1}{2\sqrt{n}}\max_{\sqrt{n}\le x\le
        \epsilon n}\exp\left( \frac{x}{2\sqrt{n}}-\sqrt{x}\right).
    \end{align*}
    For a fixed $\epsilon\in(0,1)$ and all large $n$ the maximum is
    attained at the left endpoint. Thus, the last expression converges
    to $\sqrt{e}-1<1$ as $n\to\infty$.
 \end{proof}

 \noindent \textbf{Acknowledgments.}  The author was partially
 supported by a Collaboration Grant for Mathematicians \# 209493
 (Simons Foundation).  Main results and sketches of proofs were
 originally a part of the joint project with Thomas Mountford, who
 should have been a co-author of this paper. But in accordance with
 his wishes the work is published as a single author paper.

\bibliographystyle{amsalpha}

\end{document}